\documentclass[twoside,11pt]{article}%
\usepackage{amssymb}
\usepackage{setspace}
\usepackage[labelfont=bf]{caption}
\usepackage{amssymb}
\usepackage{bm}
\usepackage{latexsym}
\usepackage{amsmath,amssymb}
\usepackage{amsthm,enumerate,verbatim}
\usepackage{amsfonts}
\usepackage{amsmath}
\usepackage{multirow}
\usepackage{graphicx}%
\usepackage{graphicx}
\usepackage{subfigure}
\usepackage{CJK}

\providecommand{\U}[1]{\protect\rule{.1in}{.1in}}
\providecommand{\U}[1]{\protect\rule{.1in}{.1in}}

\newcommand{\IR}{{\mathbb{R}}}

\newcommand{\IN}{{\mathbb{N}}}

\newcommand{\IB}{{\mathbb{B}}}

\newcommand{\BE}{\begin{equation}}
\newcommand{\EE}{\end{equation}}

\numberwithin{equation}{section}
\newtheorem{definition}{Definition}[section]
\newtheorem{theorem}{Theorem}[section]
\newtheorem{lemma}{Lemma}[section]
\newtheorem{algorithm}{Algorithm}[section]

\newtheorem{proposition}{Proposition}[section]
\newtheorem{remark}{Remark} [section]
\newtheorem{corollary}{Corollary}[section]

\begin{CJK}{GBK}{snt}
\end{CJK}

\begin{document}

\begin{CJK*}{GBK}{}

\title{\textbf{Convergence Analysis of Gradient Algorithms on Riemannian Manifolds Without Curvature Constraints and Application to Riemannian Mass}}
\date{}

\author{Chong Li\thanks{School of Mathematical Sciences, Zhejiang University, Hangzhou 310027,
P. R. China (cli@zju.edu.cn). Research of this author was supported in part by the National Natural Science Foundation of China
(grant number 11571308, 11971429) and Zhejiang Provincial Natural Science Foundation of China (grant numbers LY18A010004).}
\and Xiangmei Wang\thanks{College of Mathematics and Statistics, Guizhou University, Guiyang 550025, P. R. China (xmwang2@gzu.edu.cn). Research of this author was supported in part by the National Natural Science
Foundation of China (grant numbers 11661019) and Guizhou Provincial Natural Science Foundation of China (grant number 20161039).}
\and Jinhua Wang\thanks{Department of Mathematics,
Zhejiang University of Technology, Hangzhou 310032, P R China
 (wjh@zjut.edu.cn).  Research of this  author  was supported in part by the National Natural Science
Foundation of China (grant numbers 11771397) and Zhejiang Provincial Natural Science Foundation of China (grant number  LY17A010021).}
\and Jen-Chih Yao\thanks{China Medical University, Taichung, Taiwan 40402, R.O.C. (yaojc@mail.cmu.edu.tw). Research of this author was supported in part by the Grant MOST 108-2115-M-039-005-MY3.}
}

\maketitle

{\noindent\textbf{Abstract}} \textit{We study the  convergence issue for the gradient algorithm (employing general step sizes) for optimization problems on general Riemannian manifolds  (without curvature constraints). Under the assumption of the local convexity/quasi-convexity (resp. weak sharp minima), local/global convergence  (resp. linear convergence) results are established. As an application, the linear convergence properties of the gradient algorithm  employing the constant step sizes and the Armijo step sizes for finding the Riemannian $L^p$ ($p\in[1,+\infty)$) centers of mass are explored, respectively, which in particular extend and/or improve the corresponding results in \cite{Afsari2013}.}
%



\section{Introduction}

Let $f:M\rightarrow\overline{\IR}$ be a locally Lipschitz continuous function defined on a Riemannian manifold $M$. The following optimization problem:
\begin{equation}\label{P-1.1}
\min_{x\in M}f(x)
\end{equation}
has been  extensively studied in the literature, which not only has applications in various areas, such as computer vision, machine learning
system balancing, electronic structure computation, model reduction and robot manipulation,
low-rank approximation (see, e.g., \cite{Absil2008,Adler2002,Smith1994} and the references therein),  but also is a useful tool to treat some nonsmooth/nonconvex and/or constrained  optimization problems appeared on the Euclidean space.  As explained in  \cite{Gabay1982},  the Riemannian geometry framework can be used to decrease/overcome the difficulties caused by nonsmoothness/constaints and to enhance the performances of numerical methods by exploiting the intrinsic reduction of the dimensionality of the problem and the method's insight about the problem structure;
see also  \cite{Adler2002,Afsari2013,Bento2012M,Burke2001,Li2009,LiMWY2011,LiY2012,Wang2010} and the references therein for more details.
 One of the most typical and important examples is  the well-known  problem of finding  the Riemannian $L^p$ centers of mass of given points $\{y_i:1\le i\le N\}\subseteq M$, which can be formulated as a special case of problem \eqref{P-1.1} with the objective function $f$ defined by
\begin{equation}\label{LP-def-intro}
f(x):=
\left\{
\begin{array}{ll}
\frac{1}{p}\sum_{i=1}^{N}w_i{\rm d}^p(x,y_i),\;1\le p<+\infty\\
\max_{1\le i\le N}{\rm d}(x,y_i),\;\;p=+\infty\\
\end{array}
\right.
\quad\mbox{for any }x\in M,
\end{equation}
where $\{w_i:1\le i\le N\}\subseteq (0,+\infty)$ are the weights. This problem has various applications in the field of general data analysis, including computer graphics and animation, statistical analysis of shapes, medical imaging and sensor networks (see, e.g, \cite{Afsari2013,Groisser2004} and the references therein). As mentioned in \cite{Afsari2010}, the first study of the problem could be traced back to 1920s (the work due to Cartan) regarding the existence and uniqueness issue of the Riemannian $L^2$ centers of mass on Hadamard manifolds. After that, this problem was extensively studied in the literature, including more general existence and
uniqueness results for the Riemannian $L^p$ centers of mass and some methods for locating the Riemannian centers of mass such as
gradient algorithm, subgradient algorithm, stochastic gradient algorithm, Newton's method; see, e.g., \cite{Afsari2010,Afsari2013,Arnaudon2012,Groisser2004,YangL2010}.

Related to the optimization problem \eqref{P-1.1}, some important  notions and techniques,  such as weak sharp minima and variational analysis, have been developed in \cite{LiMWY2011,LiY2012}; while the classical numerical  methods  for solving optimization problems on the Euclidean space,  such as Newton's method, gradient algorithm, subgradient algorithm, trust region method, proximal point method, etc., have been extended to the Riemannian manifold setting; see, e.g., \cite{Absil2007,Gabay1982,Smith1994,WJHJG2015,WLLY2015,WangOptim2018}.
In the present paper,  we are particularly interested in the
gradient algorithm, which is    one of the most classical and important numerical algorithms for solving problem \eqref{P-1.1}.


The original idea of the gradient algorithm dates back to at least  the work in 1972 due to   Luenberger \cite{Luenberger1972}, where
 the gradient projection method employing the exact line search  carried out along a geodesic was proposed for solving the constrained optimization problem on the Euclidean space, that is, problem \eqref{P-1.1} with   $M:=\{x\in\IR^n:h(x)=0\}$ and $h:\IR^n\rightarrow\IR$ being also continuously differentiable;  the global (linear) convergence   results were established under the assumption that  the Hessian of the corresponding Lagrangian function $f(\cdot)+\lambda h(\cdot)$ (in the sense of the Euclidean setting) is uniformly  bounded and uniformly   positive definite on all   tangent subspaces; see \cite[Theorem 1]{Luenberger1972} for more details.
 This work was developed by Gabay in \cite{Gabay1982} with the weaker assumption  that the sub-level set of $f$ associated to $f(x_0)$ is bounded and the values of $f$ at all critical points are distinct; moreover the linear convergence rate is estimated under the assumptions that $f$ is  third continuous differentiable and that the generated sequence converges to a critical point at which the Hessian form of $f$ is positive definite (see \cite[(57)]{Gabay1982} for the definition of the Hessian form).

One important development in this direction is the work of  Smith in  \cite{Smith1994}, where he developed the gradient
algorithm (together with other algorithms such as Newton-type algorithm   and the conjugate gradient algorithm) for solving problem \eqref{P-1.1},  with $f$ being continuously  differentiable on a general Riemannian manifold. By using the pure differential geometry language (which is free from local coordinate systems), he obtained the linear convergence result for the gradient algorithm (employing the exact line search) 
in the case when the generated sequence converges to a nondegenerate point; see \cite[Theorem 2.3]{Smith1994}. 
Later, Yang studied the gradient algorithm employing
the  Armijo step sizes on a general Riemannian manifold, and established in \cite[Theorem 3.4]{Yang2007} 
the  global convergence result 
under the assumption that the generated sequence $\{x_k\}$ satisfies $\lim_{k\rightarrow+\infty}{\rm d}(x_k,x_{k+1})=0$ and has a cluster point $\bar x$ such that $\bar x$ is an isolated critical point, and in \cite[Theorem 4.1]{Yang2007}, the linear convergence
result 
under the assumption that the generated sequence
converges to a nondegenerate point.

To relax the isolatedness assumption for the cluster points of the generated sequence, the following two
   crucial assumptions   
   were introduced  in \cite{DAO1998} and  \cite{Papa2008}
   to establish the global convergence results for the gradient algorithm  (employing the Armijo step sizes) for the convex case and quasi-convex case, respectively:
\begin{description}
    \item [{\bf (A1)}]The curvatures of the Riemannian manifold $M$ are nonnegative.
  \item [{\bf (A2)}]The function $f$ is continuously  differentiable and  convex/quasi-convex  on the whole manifold $M$.\\
\end{description}
As explained in the following,  either   assumption  (A1) or (A2) is clearly too stringent.
\begin{itemize}
  \item  Assumption (A1) prevents the application to a class of Hadamard manifolds including 
  the Poincar\'{e} plane,
  hyperbolic spaces $\mathbb{H}^n$, 
  and the 
  symmetric positive definite matrix manifolds $\mathbb{R}_{++}^n$. 
  \item Assumption (A2) prevents the application to some special but important Riemannian manifolds, such as compact Stiefel manifolds ${\rm St}(p,n)$ and Grassmann manifolds ${\rm Grass}(p,n)$ ($p<n$) since there is no non-trivial (quasi-)convex function (with full domain) on a complete manifold with finite volume (see, e.g., \cite{Yau1974}).
   \item    Assumption (A1)/(A2) prevents the application  to the problem of
   the Riemannian $L^p$ centers of mass as, in general,
         the function $f$ defined by \eqref{LP-def-intro} is neither necessarily quasi-convex nor differentiable in the case when $p=1$  on the underlying Riemannian  manifolds. 
\end{itemize}


Our main purpose in the present paper is to deal with the more general case in which $M$ is not necessarily of   curvatures bounded from below
and the function $f:M\rightarrow\overline{\IR}$ is locally Lipschitz continuous on its domain (and so not necessarily continuously  differentiable, or quasi-convex/convex on the whole Riemannian manifold). More precisely, we establish the global convergence result for the gradient algorithm employing more general step sizes (which includes the Armijo step sizes   as a special case) under the following weaker assumption than (A1\& A2):
\begin{itemize}
\item The generated sequence $\{x_k\}$   has a cluster point $\bar x$ such that $\bar x$ is a critical point of $f$ and $f$ is quasi-convex around $\bar x$.
\end{itemize}
 Moreover, if  the following  assumption is additionally assumed,  we further show that the sequence $\{x_k\}$ converges linearly to a local solution:


\begin{itemize}
   \item  The cluster point  $\bar x$ is a local  weak  sharp minimizer of order $2$ for problem \eqref{P-1.1},  $f$ is convex around $\bar x$, and the step size sequence $\{t_k\}$ has a positive lower bound.
\end{itemize}
As explained before Theorem \ref{full-1} and Corollary \ref{full-AG}, the linear convergence result
 extends 
 \cite[Theorem 4.1]{Yang2007}; while the global convergence result extends/improves particularly the corresponding ones in \cite[Theorem 3.1]{Papa2008} (and so \cite[Theorem 5.3]{DAO1998}).

As an application, the convergence results for the gradient algorithm employing
the  Armijo step sizes and the constant step sizes are established, respectively, for finding the Riemannian $L^p$ centers of mass for $p\in[1,+\infty)$. We note that the (linear) convergence results for the Armijo step sizes (for $p\in[1,+\infty)$) and for the constant step sizes for $p\in[1,2)$ seem new, while the results for the constant step sizes in the case when $p\in[2,+\infty)$ extends the corresponding one in \cite[Theorem 4.1]{Afsari2013} (see the explanation before Corollary \ref{Coro-full-L-massA}).


The paper is organized as follows. As usual,  some basic notions
and notation on Riemannian manifolds, together with some related  properties about the  convexity properties of subsets and functions,   are introduced in the next section. Main results,  including the local/global/linear convergence properties of the gradient algorithm on general manifolds,  are presented in section 3, and the application to the Riemannian $L^p$ centers of mass is provided in the last section.

\section{Notation and preliminary results}
Notation and terminologies used in the present paper are standard;
the readers are referred to some textbooks for more details; see, e.g.,
\cite{Carmo1992,Sakai1996,Udriste1994}.

Let $M$ be a connected and complete $n$-dimensional Riemannian manifold. We use
$\nabla$ to denote the Levi-Civita connection on $M$. Let $x\in M$,
and let ${\rm T}_{x}M$ stand for the tangent space at $x$ to $M$. We denote by
$\langle,\rangle_{x}$ the scalar product
 on $T_{x}M$ with the associated norm $\|\cdot\|_{x}$, where the subscript $x$ is
sometimes omitted. For $y\in{M}$, let $\gamma:[0,1]\rightarrow M$ be
a piecewise smooth curve joining $x$ to $y$. Then, the arc-length of
$\gamma$ is defined by $l(\gamma):=\int_{0}^{1}\|{\gamma}'(t)\|dt$,
while the Riemannian distance from $x$ to $y$ is defined by ${\rm
d}(x,y):=\inf_{\gamma}l(\gamma)$, where the infimum is taken over
all piecewise smooth curves $\gamma:[0,1]\rightarrow M$ joining $x$
to $y$. The closed metric ball and the open metric ball centered at $x$ with radius
$r$ are denoted by $\mathbb{B}(x,r)$ and $\mathbb{U}(x,r)$, respectively,
that is,
$$
\mathbb{B}(x,r):=\{y\in M:{\rm d}(x,y)\leq r\}\;\;\mbox{and}\;\; \mathbb{U}(x,r):=\{y\in M:{\rm d}(x,y)< r\}.
$$

A vector field $V$ is said to be parallel along $\gamma$  if
$\nabla_{{\gamma}'}V=0$. In particular, for a smooth curve $\gamma$,
if ${\gamma}'$ is parallel along itself, then  $\gamma$ is called a
geodesic; thus, a smooth curve $\gamma$ is a geodesic if and only
if $\nabla_{{\gamma}'}{{\gamma}'}=0$. A geodesic
$\gamma:[0,1]\rightarrow M$ joining $x$ to $y$ is minimal if its
arc-length equals its Riemannian distance between $x$ and $y$. By
the Hopf-Rinow theorem \cite{Carmo1992}, $(M,{\rm d})$ is a complete
metric space, and there is at least one minimal geodesic joining $x$
to $y$ for any points $x$ and $y$.

Let $Q\subseteq M$ be a subset. As usual, we use $\overline{Q}$ and $\partial Q$ to stand for the closure and the boundary of a subset $Q\subseteq M$, respectively. The distance function ${\rm d}_{Q}(\cdot)$ associated to $Q$ and the projection $P_Q(\cdot)$ onto
$Q$ are respectively defined by, for any $x\in M$,
\[
{\rm d}_{Q}(x):=\inf_{y\in Q}{\rm d}(x,y)\quad\mbox{and}\quad P_Q(x):=\left\{y\in Q:{\rm d}(x,y)={\rm d}_{Q}(x)\right\}.
\]
Given points  $x,\,y\in Q$, the set of
all geodesics $\gamma:[0,1]\rightarrow M$ with $\gamma(0)=x$ and
$\gamma(1)=y$ satisfying $\gamma([0,1])\subseteq Q$ is denoted by   $\Gamma^Q_{xy}$, that is,
$$\Gamma^Q_{xy}:=\{\gamma:[0,1]\rightarrow Q:\;\gamma(0)=x,\, \gamma(1)=y\mbox{ and } \nabla_{{\gamma}'}{\gamma}'=0\}.
$$
In particular, we write  $\Gamma_{xy}$ for $\Gamma^M_{xy}$.
Two important  structures on $M$  will be used frequently in our study: one is the exponential map
$\exp_x:{\rm T}_xM\rightarrow M$,  
and the other  is the parallel transport along the geodesic $\gamma\in \Gamma_{xy}$ 
denoted by $P_{\gamma,y,x}$. 
For simplicity, we will
write $P_{y,x}$ for $P_{\gamma,y,x}$ if $\gamma\in \Gamma_{xy}$ is the unique minimal geodesic  and no confusion arises.

Recall two constants related to  a point $x\in M$: the injectivity
radius $r_{\rm inj}(x)$ and the convexity radius $r_{\rm cvx}(x)$ of $x$, which are defined by
$$
r_{\rm inj}(x):=\sup\left\{r>0:\exp_x(\cdot)
%
\mbox{ is a diffeomorphism on }\IB(0,r)\subset {\rm T}_xM
\right\}
$$
 and
%
\begin{equation}\label{convexity-radius}
r_{\rm cvx}(x):=\sup\left\{r>0:\begin{array}{ll}&\mbox{each ball in }\IB(x,r) \mbox{ is strongly convex}\\
&\mbox{and each geodesic in } \IB(x,r)\mbox{ is minimal}\end{array}\right\},
\end{equation}
 respectively.
Then, $r_{\rm inj}(x)\ge r_{\rm cvx}(x)>0$ for any $x\in M$; see, e.g., \cite[Theorem 5.3]{Sakai1996}.
 In particular, $r_{\rm inj}(x)=r_{\rm cvx}(x)=+\infty$ for each $x\in M$  if $M$ is a Hadamard manifold.
Moreover, for any compact subset $Q\subseteq M$, we have that
$$
r_{\rm inj}(Q):=\inf \{r_{\rm inj}(x):{x\in Q}\}>0\quad\mbox{and}\quad r_{\rm cvx}(Q):=\inf \{r_{\rm cvx}(x):{x\in Q}\}>0;
$$
see \cite[Theorem 5.3, p. 169]{Sakai1996} or \cite[Lemma 3.1]{LiLi2009}.


Definition \ref{convexset} below presents the notions of different kinds of convexities about subsets in $M$;
see e.g., \cite{LiLi2009,Wang2010}. 

\begin{definition}\label{convexset} 
A subset $Q\subseteq M$ is said to be

{\rm (a)} weakly convex if, for any $x,y\in Q$, there is a minimal
geodesic of $M$ joining $x$ to $y$ and it is in $Q$;

{\rm (b)} strongly convex if, for any $x,y\in Q$, there is just one
minimal geodesic of $M$ joining $x$ to $y$ and it is in $Q$;

{\rm (c)} totally convex if, for any $x,y\in Q$, all geodesics of $M$ joining $x$ to $y$ lie in $Q$;

\end{definition}

Note by definition that the  strongly/totally convexity  implies the weakly convexity for any subset $Q$, and note also
  that $Q$ is weakly convex if and only if so is $\overline{Q}$. 

Consider now a proper real-valued function $f:M\rightarrow\overline{\IR}:=(-\infty,\infty]$ with its domain denoted by 
$\mathcal{D}(f)$. 
Letting $k\in \IN$, we use $\mathcal{D}^k(f)$ to denote the set of all points $x\in \mathcal{D}(f)$ at which $f$ is  $k{\rm th}$  differentiable, that is, \begin{equation}\label{cddom}
\mathcal{D}^k(f):=\{x\in {\rm int}\mathcal{D}(f):\mbox{$f$ is $k{\rm th}$  differentiable at $x$}\}.
\end{equation}
As usual, we say that $f$ is 
$C^k$ on $Q$ if $Q\subseteq\mathcal{D}^k(f)$ and its $k{\rm th}$  derivative 
is continuously at each point of $Q$, and that  $f$ is $C^2$ around $\bar x$ if it is $C^2$ on $\IB(x,r)$ for some $r>0$. 
The gradient (resp. the Hessian) of $f$ at $x\in  \mathcal{D}^1(f)$ (resp.  $x\in  \mathcal{D}^2(f)$) is denoted by $\nabla f(x)$ (resp. $\nabla^2 f(x)$).
Recall that the gradient field $\nabla f$ is  Lipschitz continuous around $\bar x\in{\rm int}\mathcal{D}^1(f)$, if there exist positive constants $\delta, L$ (with $\delta \le r_{\rm cvx}({\bar x})$) such that
$$
\|\nabla f(x)-P_{x,y}\nabla f(y)\|\le L{\rm d}(x,y)\quad\mbox{for any }x,y\in \IB(\bar x,\delta).
$$
Thus, if $f$ is $C^2$ around $\bar x$, then $\nabla f$ is  Lipschitz continuous around $\bar x$.

Item (b) in the following definition was known in \cite[Definition 6.1 (b)]{LiMWY2011} (for the convexity) and \cite[Definition 2.2]{Papa2008}
(for the quasi-convexity in the case when $\mathcal{D}(f)=M$). 

\begin{definition}\label{convexfunction}
Let $f:M\rightarrow\overline{\IR}$ be 
proper 
and let $Q\subseteq\mathcal{D}(f)$ be weakly convex.  Then, $f$ is said to be

{\rm (a)} convex (resp. quasi-convex) on $Q$ if,
for any $x, y\in Q$ and any geodesic $\gamma\in
\Gamma^Q_{xy}$, the composition $f\circ\gamma:[0,1]\rightarrow\IR$ is convex (resp. quasi-convex) on $[0,1]$;

{\rm (b)} 
%
convex  (resp. quasi-convex) if $\mathcal{D}(f)$ is weakly convex and $f$  is   convex  (resp. quasi-convex) on $\mathcal{D}(f)$.

{\rm (c)}  convex  (resp. quasi-convex) around $x\in\mathcal{D}(f)$ if 
$f$ is convex  (resp. quasi-convex) on $\IB(x,r)$ for some $r>0$.
\end{definition}

It is clear that the convexity implies the quasi-convexity. The assertions in the following lemma can be proved directly by definition and are known for some special cases; see. e.g., \cite[Theorems 5.1, 6.2]{Udriste1994} for assertions (i), (iii) and \cite[Proposition 3.1]{Nemeth1998} for assertion (ii).

\begin{lemma}\label{QC-F}
Let $f:M\rightarrow\overline{\IR}$ be proper. Let $Q\subseteq\mathcal{D}(f)$ be weakly convex and let $x\in Q\cap \mathcal{D}^1(f)$. 
Then, the following assertions hold.

{\rm (i)} If $f$ is convex on   $Q$, then it holds  for any 
$y\in Q$ that
$$
f(y)\ge f(x)+\langle\nabla f(x),\gamma_{xy}'(0)\rangle\quad\mbox{for all }\gamma_{xy}\in\Gamma_{xy}^Q.
$$

{\rm (ii)}
If $f$ is quasi-convex on   $Q$, then it holds for any 
$y\in Q$ with $f(y)\le f(x)$ that
$$
\langle\nabla f(x),\gamma_{xy}'(0)\rangle\le0\quad\mbox{for all }\gamma_{xy}\in\Gamma_{xy}^Q.
$$

{\rm (iii)} If     $f$ is $C^2$ on $Q$, then $f$ is convex on $Q$ if and only if $\nabla^2f(x)$ is semi-positive definite for each  $x\in Q$.

\end{lemma}

We show in the following lemma some inequalities, which play important roles in our study. For this purpose, we define the function
$\hbar:[0,+\infty)\rightarrow\mathbb{R}$ as in \cite{WangOptim2018} by
\begin{equation}\label{tanhdef}
\hbar(t):=\left\{\begin{array}{ll}\frac{\tanh t}{t}&\quad\mbox{ if }t\in(0,\infty),\\
1&\quad \mbox{ if }t=0.
\end{array}\right.
\end{equation}
Note that, $\hbar$ is continuous and decreasing monotonically on $[0,+\infty)$.

\begin{lemma}\label{basic-QC}
Let $f:M\rightarrow\overline {\IR}$ be proper, 
and  $Q\subseteq M$ be  weakly convex such that
$Q_f:=\mathcal{D}(f)\cap Q$ is weakly convex with nonempty interior (i.e., ${\rm int}Q_f\neq\emptyset$) and its sectional curvatures are bounded from below by  $\kappa\le 0$.
 Let $t\ge 0,\,z\in {\rm int}Q_f,\,\,x\in {\rm int}Q_f\cap\mathcal{D}^1(f)$ and
$\gamma:[0,+\infty)\rightarrow M$ be the geodesic such that
\begin{equation}\label{gsf0}
\gamma(0)=x,\quad \gamma'(0)=-\nabla f(x)\not=0\quad
\mbox{and} \quad \gamma([0,t])\subset {\rm int}Q_f.
\end{equation}
Then, the following assertions hold:

{\rm (i)} If $f$ is convex on $Q_f$, then the following inequality holds:
\begin{equation}\label{basic01}
\begin{array}{ll}
{\rm d}^2(\gamma(t),z)\leq
{\rm d}^2(x,z)
 +\frac{2\sinh(\sqrt{|\kappa|}t\|\nabla f(x)\|)}{\sqrt{|\kappa|}\|\nabla f(x)\|\hbar\left(\sqrt{|\kappa|} {\rm
d}(x,z)\right)}\left(\frac{t\|\nabla f(x)\|^2}{2}-\hbar\left(\sqrt{|\kappa|} {\rm
d}(x,z)\right)\left(f(x)-f(z)\right)\right).
\end{array}
\end{equation}

{\rm (ii)} If $f$ is quasi-convex on $Q_f$ and $f(z)\le f(x)$, then the following inequalities hold:
\begin{equation}\label{basic-QC-0}
\cosh\left(\sqrt{|\kappa|}{\rm d}(\gamma(t),z)\right)\leq
\cosh\left(\sqrt{|\kappa|}{\rm d}(x,z)\right)\left(1+\frac{|\kappa|}{2}t\|\nabla f(x)\|\sinh(t\|\nabla f(x)\|)\right);
\end{equation}
\begin{equation}\label{basic-QC-1}
\begin{array}{ll}
{\rm d}^2(\gamma(t),z)<
{\rm d}^2(x,z)
 +\frac{3 t^2\|\nabla f(x)\|^2}{2\hbar\left(\sqrt{|\kappa|} {\rm
d}(x,z)\right)}\quad\mbox{if }\sqrt{|\kappa|}t\|\nabla f(x)\|\le 1.
\end{array}
\end{equation}
\end{lemma}

\begin{proof} (i)
We note that the comparison theorem for a generalized hinge
introduced in \cite[p. 161,
Theorem 4.2]{Sakai1996} is still true  with $Q$ in place of $M$, provided that 
${\rm int}Q$ contains the corresponding generalized hinge. 
Thus, the argument for proving  
\cite[Lemma 3.2]{WLWYSIAM2015} and \cite[Lemma 3.1]{WLYJOTA2015} remains  valid.   
Hence, assertion {\rm (i)} holds by (3.6) in \cite[Lemma 3.2]{WLWYSIAM2015} (applied to  $\{f\}$,  $t\|\nabla f(x)\|$ in place of $\{f_i\}$, $t$ (noting that $\partial f(x)=\{\nabla f(x)\}$)).

{\rm (ii)}   Suppose that $f$ is quasi-convex on $Q_f$ and $f(z)\le f(x)$, and  let $\gamma_{xz}\in\Gamma_{xz}^{Q_f}$ be a minimal geodesic joining $x$ and $z$. Without loss of generality, we assume that $\kappa=-1$. Then, we have from \cite[(9)]{WLYJOTA2015} (applied to $x,\gamma(t)$, $t\|\nabla f(x)\|$  in place of $x^k,x^{k+1}$, $t_k$) that
\begin{equation}\label{yidiyao}
\begin{array}{lll}
\cosh{\rm d}(\gamma(t),z)\leq&\cosh{\rm d}(x,z)+\cosh{\rm d}(x,z)\sinh
(t\|\nabla f(x)\|)\left(\frac{t\|\nabla f(x)\|}{2}-\tanh{\rm d}(x,z)\cos \alpha\right),
\end{array}
\end{equation}
 where $\alpha:=\angle_x(\gamma,\gamma_{xz})$ is the angle between $\gamma$ and $\gamma_{xz}$ at $x$.
Below, we verify that $\cos\alpha\ge 0$. Granting this, \eqref{basic-QC-0} follows immediately from \eqref{yidiyao}.
To do this, we note by Lemma \ref{QC-F}(ii) (applied to  $Q_f$, $z$ in place of $Q$, $y$)  that $\langle\nabla f(x),{\gamma}'_{xz}(0)\rangle\le0$, and so $ \langle{\gamma}'(0),{\gamma}'_{xz}(0)\rangle=-\langle\nabla f(x),{\gamma}'_{xz}(0)\rangle\ge 0$, thanks to \eqref{gsf0}.
Thus, by definition,
%
 $
 \cos\alpha=\frac{\langle{\gamma}'(0),{\gamma}'_{xz}(0)\rangle}{\|{\gamma}'(0)\|\cdot{\rm d}(x,z)}\ge0
$ as desired to show.

To show \eqref{basic-QC-1}, assume $t\|\nabla f(x)\|\le 1$  and note  that $\sinh s< \frac{3}{2}s$ holds for any $s\in (0,1]$ (which could be easily checked by elementary calculus). Then,  \eqref{basic-QC-0} implies  that
$$
\cosh\left({\rm d}(\gamma(t),z)\right)\leq
\cosh\left({\rm d}(x,z)\right)\left(1+\frac34t^2\|\nabla f(x)\|^2\right).
$$
Therefore, in view of the definition of $\hbar$ in \eqref{tanhdef}, \eqref{basic-QC-1} is seen to hold from
the following estimate (see \cite[Lemma 3.1]{WLWYSIAM2015}):
$$\mbox{$\cosh s_1-\cosh s_2\ge \frac{(s_1^2-s_2^2)\sinh s_2}{2s_2}$ \quad for any  $s_1,s_2\in (0,+\infty)$}.$$
  The proof is complete.
\end{proof}

We shall use the following known lemmas in what follows; see, e.g., \cite[lemma 2.3]{WangOptim2018} for Lemma \ref{lemma3} and  \cite{Ermolev1969} for Lemma \ref{fejer2}.

\begin{lemma}\label{lemma3}
Let $\{a_{k}\}$, $\{b_k\}\subset (0,+\infty)$ be   sequences such that $\sum_{k=0}^{\infty}b_k<\infty$ and
$a_{k+1}\leq a_{k}(1+b_k)$ for each  $k\in\IN$.
  Then, $\{a_{k}\}$ is convergent and so it is bounded.
\end{lemma}


\begin{lemma}\label{fejer2}
Let $\{y_k\}\subset M$ be a sequence 
quasi-Fej\'{e}r convergent to $S$, namely there exists a sequence
$\{\varepsilon_{k}\}\subset (0,+\infty)$ satisfying
$\sum_{k=1}^{\infty}\varepsilon_{k}<\infty$ such that ${\rm d}^{2}(y_{k+1},z)\leq {\rm d}^{2}(y_{k},z)+\varepsilon_{k}$   for any  $k\in\IN$ and   $z\in S$. Then, $\{y_k\}$ is
bounded. Furthermore, if  $\{y_k\}$ has a cluster point $\bar y$ which
belongs to $S$, then $\lim_{k\rightarrow\infty}y_k=\bar y$.
\end{lemma}


\section{Gradient algorithm}
As in Section 1,   $f:M\rightarrow\overline{\IR}$ is a proper locally Lipschitz continuous function. Associated to the optimization problem \eqref{P-1.1}, 
let $C_f$  denote the set of all critical points of $f$: 
$$C_f:=\{x\in\mathcal{D}^1(f): \nabla f(x)=0\},$$
where $\mathcal{D}^1(f)$ is the set defined by \eqref{cddom}.
We always  assume for the remainder  that 
\begin{equation}\label{BlanketA}
\bar f:=\inf_{x\in M}f(x)>-\infty\quad \mbox{and}\quad {\rm int}\mathcal{D}(f)\not=\emptyset.
\end{equation}

We begin with 
the following gradient algorithm for solving problem \eqref{P-1.1}. 

\begin{algorithm}\label{GDA} Give  $x_0\in\mathcal{D}(f)$, $\beta\in(0,1)$, $R\in[1,+\infty)$ and set $k:=0$.

\noindent {\rm{Step 1}}. If 
$x_k\in C_f$ or $x_k\notin\mathcal{D}^1(f)$, 
then stop;
otherwise construct the geodesic
$\gamma_{k}$ such that
\begin{equation}\label{GDA-1}
 \gamma_{k} (0)=x_{k}\quad\mbox{and}\quad
 \gamma'_k(0)=-\nabla f(x_k).
\end{equation}

\noindent {\rm{Step 2}}. Select the step size $t_k\in(0,R]$ satisfies the following inequality:
\begin{equation}\label{GDA-2}
f(\gamma_k(t_k))\le f(x_k)-\beta t_k\|\nabla f(x_k)\|^2.
\end{equation}

\noindent {\rm{Step 3}}. Set $x_{k+1}:=\gamma_{k}(t_{k})$, replace $k$ by $k+1$ and go to step 1.
\end{algorithm}

\begin{remark}\label{rem3.1}
Let $\{x_k\}$ be a sequence generated by Algorithm \ref{GDA}  with initial point $x_0\in\mathcal{D}(f)$. Then, by   Algorithm \ref{GDA},  the following inequalities hold for any $k\in\IN$:
\begin{equation}\label{D-tpp}
\mbox{${\rm d}(\gamma_k(t),x_k)
\le t\,\|\nabla f(x_k)\|$\quad for any $t\in [0,t_k]$};
\end{equation}
\begin{equation}\label{Full-p1}
\sum_{j=0}^k t_j\|\nabla f(x_j)\|^2\le \frac{f(x_0)-f(x_{k+1})}{\beta}\le \frac{f(x_0)-\bar f}{\beta}<+\infty
\end{equation}
by the blanket assumption \eqref{BlanketA}. 
In particular, one has that       $t_k\|\nabla f(x_k)\|\to 0$.
\end{remark}


Recall that Algorithm \ref{GDA} is said to employ 
the Armijo step sizes if each step size $t_k$ in Step 2 is chosen by
\begin{equation}\label{GDA-2-A}
t_k:=\max\{2^{-i}:i\in\IN,\; f(\gamma_k(2^{-i}))\le f(x_k)-\beta2^{-i}\|\nabla f(x_k)\|^2\};
\end{equation}
 see, e.g.,\cite{Gabay1982,Smith1994}.   Note that \eqref{GDA-2-A} particularly implies \eqref{GDA-2}.

The following remark regards  the well definedness and   the  partial convergence property of Algorithm \ref{GDA}.
\begin{remark}\label{remark-partial02}
{\rm (a)} Suppose that $\{x_j:0\le j\le k\}\subset \mathcal{D}(f)$ is  generated by Algorithm \ref{GDA} such  that $x_k\in \mathcal{D}^1(f)$ is not a critical point of $f$. Then, using the argument as one did for proving \cite[Proposition 3.1]{Yang2007}, we can check that 
\eqref{GDA-2-A} is well defined.  Therefore,  if each generated iterate $x_k\in \mathcal{D}^1(f)$ (e.g., $\mathcal{D}^1(f)=\mathcal{D}(f)$), 
then Algorithm \ref{GDA} employing the Armijo step sizes is well defined. 

{\rm (b)} Let $\{x_k\}$ be a sequence generated by Algorithm \ref{GDA} employing step sizes $\{t_k\}$ with a positive lower bound or employing   the Armijo step sizes. Then, any   cluster point $\bar x$ of $\{x_k\}$  such that $\nabla f$ is continuous at $\bar x$  is a critical point of $f$, that is, $\bar x\in C_f$. Indeed, it is immediate from Remark \ref{rem3.1} for the case when  Algorithm \ref{GDA} employs the step size  $\{t_k\}$ with  a positive lower bound; while for the case when Algorithm \ref{GDA} employs the Armijo step sizes it can be checked by the argument as one did   for proving  \cite[Corollary 3.1]{Yang2007}.

\end{remark}

%

\subsection{Local convergence and Linear convergence}


We shall consider the local convergence and the linear convergence of Algorithm {\ref{GDA}} in this subsection. For this purpose,  consider the following assumption:
\begin{equation}\label{assumption-a}
\mbox{
  $\bar x\in C_f\cap{\rm int}\mathcal{D}^1(f)$, and  $\nabla f$ is continuous at $\bar x$.}
\end{equation}
For the following key lemma, recall that $R$ is the constant given at the beginning of Algorithm {\rm\ref{GDA}}.

\begin{lemma}\label{lemma1-LC} Suppose  that assumption \eqref{assumption-a} holds.
 Then, for any $\delta>0$, there exist $\bar \delta>0$ and $\bar c\ge3$ satisfying $\bar c\bar \delta<\delta$ such that,  for any $x_0\in\IB(\bar x,\bar\delta)$ and $k\in\IN$,
if  $\{x_{j}:0\le j\le k+1\}$ is generated by Algorithm {\rm\ref{GDA}} to satisfy  $\{x_{j}:0\le j\le k\}\subset \IB(\bar x, \bar c\bar \delta)$,
 then one has  the following assertions for each $z\in \IB(\bar x,\bar c\bar\delta)$ satisfying $f(z)\le f(x_{k+1})$:

{\rm (i)} If $f$ is convex around $\bar x$, then
\begin{equation}\label{L-L30}
{\rm d}^2(x_{k+1},z)\le
{\rm d}^2(x_k,z)+\frac{8\sinh(\sqrt{|\kappa|}t_k\|\nabla f(x_k)\|)}{3\sqrt{|\kappa|}\|\nabla f(x_k)\|}\left( \frac{1}{2\beta}-\hbar( 2\sqrt{|\kappa|}\bar c\bar\delta)\right)(f(x_k)-f(z)).
\end{equation}

{\rm (ii)} If $f$ is quasi-convex around $\bar x$, 
then
\begin{equation}\label{lemma1-LC00c}
{\rm d}^2(x_{k+1},z)
\leq
{\rm d}^2(x_{k},z)+2Rt_k\|\nabla f(x_k)\|^2\le {\rm d}^2(x_0,z)+\delta {\rm d}(x_0,z);
\end{equation}
in particular, $ x_{k+1}\in \IB(\bar x,\bar c\bar \delta)$ if $f(\bar x)\le f(x_{k+1})$.

\end{lemma}

\begin{proof}
Noting that any closed ball is compact, we have   by \cite[P.166]{Bishop1964} that the curvatures of the ball   $\IB(\bar x,r_{\rm cvx}({\bar x}))$ are bounded, where   $r_{\rm cvx}({\bar x})$ is the convexity radius at $\bar x$     defined in \eqref{convexity-radius}.
Let  $\kappa\le0$ be   a lower bound of the curvatures of $\IB(\bar x,r_{\rm cvx}({\bar x}))$. 
For simplicity, we may assume, without loss of generality, that $\kappa=-1$.
Thanks to assumption \eqref{assumption-a}, there exits  $\delta>0$   such that
\begin{equation}\label{jiash0}
\IB(\bar x,\delta)\subset \mathcal{D}^1(f),\;
\;\; \delta<\min\left\{r_{\rm cvx}({\bar x}),   \frac1{\sqrt{|\kappa|}}\right\}\;\;\; \mbox{and}\;\;\;
    R\sqrt{|\kappa|}\|\nabla f(\cdot)\|\le 1 \mbox{ on } \IB(\bar x,\delta).
\end{equation}
 We further choose
 $0< \delta_1< \delta/2$ such that
\begin{equation}\label{local-diameter-1}
\|\nabla f(x)\|\le \frac{\beta\delta}{2R}\quad\mbox{for any }x\in\IB(\bar x,\delta_1),
\end{equation}
and define
\begin{equation*}\label{Deltaremark}
\bar\delta:=\frac{\delta_1^3}{\delta^2}\quad\mbox{and}\quad \bar c:=\sqrt{1+\left(\frac{\delta}{\delta_1
}\right)^3}.
\end{equation*}
Then,  $\bar\delta\le \delta_1\le \frac\delta2$ and $\bar c\bar\delta=\delta_1 \sqrt{\left(\frac{\delta_1}{\delta}\right)^4+\frac{\delta_1}{\delta}}$; thus one has that  $\bar c\ge 3$ and
$\bar c\bar\delta< \delta_1<\frac{\delta}{2}$. 
To proceed, let $x_0\in\IB(\bar x,\bar\delta)$ and let $k\in\IN$ be such that
$\{x_{j}:0\le j\le k+1\}$ is generated by Algorithm {\rm\ref{GDA}} and $\{x_{j}:0\le j\le k\}\subset \IB(\bar x, \bar c\bar \delta)$. Fix  $j\in\{0,1,\dots,k\}$, and let $\gamma_{j}$  be the geodesic determined by  \eqref{GDA-1}.
Then,  by \eqref{local-diameter-1},   $\|\nabla f(x_j)\|\le \frac{\beta\delta}{2R}$ (as $\bar c\bar\delta<\delta_1$), and it follows from \eqref{D-tpp} that, for any $t\in [0,t_j]$,
$${\rm d}(\gamma_j(t),\bar x)\le {\rm d}(\gamma_j(t),x_{j})+{\rm d}(x_{j},\bar x)< t\,\|\nabla f(x_j)\|+\bar c\bar\delta\le\frac{\beta\delta}{2}+\frac{1}{2}\delta<\delta,$$
Hence, one has that
$$
 \gamma_j({[0, t_j]})\subseteq {\rm int}\IB(\bar x,\delta)\subseteq \IB(\bar x,r_{\rm cvx}({\bar x}))\cap \mathcal{D}^1(f).
$$
Thus, noting that   $\IB(\bar x,\delta)$ is strongly convex by the second one of \eqref{jiash0}, Lemma \ref{basic-QC}  is applicable (to $\IB(\bar x,\delta)$, $t_{j}$, $x_{j}$, $\gamma_{j}$,
in place of $Q$, $t$, $x$, $\gamma$). 
 Now let   $z\in \IB(\bar x,\bar c\bar\delta)$ be such that  $f(z)\le f(x_{k+1})$. Then, noting that $\{f(x_{j})\}$ is decreasing monotonically
and using \eqref{GDA-2}, we check that
 \begin{equation}\label{newm}
  f(z)\le f(x_{k+1})\le f(x_{j+1})\quad\mbox{and}\quad  t_j\|\nabla f(x_j)\|^2\le\frac{f(x_j)-f(x_{j+1})}{\beta} \le \frac{f(x_j)-f(z)}{\beta}.
 \end{equation}
Moreover, we have that
$$
{\rm d}(x_j,z)\le {\rm d}(x_j,\bar x)+{\rm d}(z,\bar x)\le 2\bar c\bar\delta \quad\mbox{and}\quad\sqrt{|\kappa|}{\rm d}(x_j,z)<\frac12,
$$
because $x_j, z\in \IB(\bar x,\bar c\bar\delta)$  and   $\sqrt{|\kappa|}{\rm d}(x_j,z)\le \sqrt{|\kappa|} \bar c\bar\delta  \le \frac12 $  (as   $\bar c\bar\delta<\frac{\delta}{2}<1 $ by   \eqref{jiash0}). Therefore,
\begin{equation}\label{3.14-12}
\hbar\left(\sqrt{|\kappa|}{\rm
d}(x_{j},z)\right)\ge   \hbar( 2\sqrt{|\kappa|}\bar c\bar\delta)\ge \hbar (1)>\frac34.
\end{equation}
This,  together with \eqref{newm}, implies that
 \begin{equation}\label{L-L2}
\frac{t_j\|\nabla f(x_j)\|^2}{2}-\hbar\left( \sqrt{|\kappa|}{\rm
d}(x_j,z)\right)\left(f(x_j)-f(z)\right)
 \le\left( \frac{1}{2\beta}-\hbar( 2\sqrt{|\kappa|}\bar c\bar\delta)\right)(f(x_j)-f(z)).
\end{equation}
Thus, if   $f$ is convex around $\bar x$,  then we may assume, without loss of generality, that  $f$ is convex on $ \IB(\bar x,\delta)$
(using a smaller $\delta$ if necessary); hence, by using \eqref{3.14-12} and \eqref{L-L2} (with $j=k$),  \eqref{L-L30} follows from    inequality \eqref{basic01} (noting $x_{k+1}=\gamma_k(t_{k})$), showing
assertion (i). 

To show assertion  (ii), assume that $f$ is quasi-convex around $\bar x$. Then,   $f$ is quasi-convex on
$\IB(\bar x,\delta)$ as explained earlier (using a smaller $\delta$ if necessary).
 Since
$f(z)\le  f(x_j)$   
and $\sqrt{|\kappa|}t_j\|\nabla f(x_j)\|\le\sqrt{|\kappa|}R\|\nabla f(x_j)\|\le 1$ by the third item of \eqref{jiash0}, 
it follows from \eqref{basic-QC-1} that 
\begin{equation}\label{local-b1lc}
{\rm d}^2(\gamma_{j}(t_j),z)\leq
{\rm d}^2(x_{j},z)
 +\frac{3t_j^2\|\nabla f(x_j)\|^2}{2\hbar\left(\sqrt{|\kappa|}{\rm
d}(x_{j},z)\right)}
\leq
{\rm d}^2(x_{j},z)+2Rt_j\|\nabla f(x_j)\|^2,
\end{equation}
where the last inequality holds by \eqref{3.14-12} (recalling   $t_j\le R$). 
Summing up the   inequalities in \eqref{local-b1lc} over $0\le j\le k-1$, we have that
$$
{\rm d}^2(x_{k},z)+t_k\|\nabla f(x_k)\|^2\leq
{\rm d}^2(x_{0},z)
+2R\sum_{j=0}^{k}t_j\|\nabla f(x_j)\|^2
$$
(noting that $x_{j+1}=\gamma_{j}(t_j)$). Since $\sum_{j=0}^kt_j\|\nabla f(x_j)\|^2\le\frac{f(x_0)-f(x_{k+1})}{\beta}\le \frac{  f(x_0)-f(z)}{ \beta}$ by \eqref{Full-p1} (recalling  $f(z)\le f(x_{k+1})$), it follows  that
 \begin{equation}\label{local-b2c}
{\rm d}^2(x_{k},z)+t_k\|\nabla f(x_k)\|^2\le {\rm d}^2(x_0,z)+\frac{2R}{\beta}\left(f(x_0)-f(z)\right).
 \end{equation}
Recalling $x_0,z\in \IB(\bar x,\bar c\bar \delta)\subset\IB(\bar x,\delta_1)\subset \IB(\bar x,r_{\rm cvx}({\bar x}))\cap \mathcal{D}^1(f)$, the unique  minimal geodesic $\gamma$ joining $x_0$ to $\bar x$ is in $\IB(\bar x,\delta)$, and then we can apply
the mean value theorem   to choose  $\xi\in(0,1)$ such that
$$f(x_0)-f(z)\le\|\nabla f(\gamma(\xi))\|{\rm d}(x_0,z)\le \frac{\beta\delta}{2R}{\rm d}(x_0,z),$$
where the last inequality is from \eqref{local-diameter-1}.
This, together with \eqref{local-b1lc} (with $j=k$) and
\eqref{local-b2c}, implies \eqref{lemma1-LC00c}. Particular, if $f(\bar x)<f(x_{k+1})$, then we estimate by \eqref{lemma1-LC00c} (applied to $\bar x$ in place of $z$, and noting ${\rm d}(x_0,\bar x)\le \bar\delta$)  that
$$
{\rm d}^2(x_{k+1},\bar x)< \bar\delta^2 + \delta\bar\delta=\bar\delta^2(1+\frac{\delta}{\bar\delta})={(\bar c\bar \delta)^2},
$$
showing that $x_{k+1}\in\IB(\bar x,\bar c\bar\delta)$, and so assertion (ii) is proved.
The proof is complete.
\end{proof}

\begin{remark}\label{remark-local01} In addition to assumption \eqref{assumption-a} made in Lemma \ref{lemma1-LC}, assume further that  $\bar x\in M$ is a local minimizer of $f$ and that   $f$ is quasi-convex around $\bar x$. Then,
%
for any $\delta>0$, there exist $\bar \delta>0$ and $\bar c\ge3$ satisfying $\bar c\bar \delta<\delta$ such that, for any $x_0\in\IB(\bar x,\bar\delta)$ and $k\in\IN$, if  $\{x_{j}:0\le j\le k+1\}$ is generated by Algorithm {\rm\ref{GDA}} to satisfy  $\{x_{j}:0\le j\le k\}\subset \IB(\bar x, \bar c\bar \delta)$,
then there holds that
\begin{equation}\label{remarklm0}
\mbox{$f(\bar x)\le f(x_{k+1})$\quad and \quad $x_{k+1}\in\IB(\bar x, \bar c\bar \delta)\subseteq \mathcal{D}^1(f)$.}
\end{equation}
Indeed, one can choose at the beginning of the proof of Lemma \ref{lemma1-LC}
$\delta>0$ and $0< \delta_1< \delta/2$ 
small enough so that  \eqref{jiash0}, \eqref{local-diameter-1} and the following condition   hold:
\begin{equation}\label{remarklm1}
f(\bar x)\le f(x)\quad \mbox{for any } x\in\IB(\bar x, \delta). 
\end{equation}
Thus, 
if $x_0\in\IB(\bar x,\bar\delta)$ and $k\in\IN$, and if $\{x_{j}:0\le j\le k+1\}$ is generated by Algorithm {\rm\ref{GDA}} so  that $\{x_{j}:0\le j\le k\}\subset \IB(\bar x, \bar c\bar \delta)$, then one has by  \eqref{local-diameter-1}  that
$t_k\|\nabla f(x_k)\|\le R \frac{\beta \delta}{2R}\le\frac\delta2$ (as $\beta<1$) and so by \eqref{D-tpp} that
 $${\rm d}(\bar x, x_{k+1})\le {\rm d}(\bar x, x_{k})+{\rm d}(x_k, x_{k+1})\le \bar c\bar\delta+t_k\|\nabla f(x_k)\|<\frac{\delta}{2}+\frac{\delta}{2}= \delta
$$
because  
 $\bar c\bar\delta< \delta_1<\frac{\delta}{2}$ as   noted in the line after \eqref{Deltaremark}, 
which, together with \eqref{remarklm1}, implies that
  $f(\bar x)\le f(x_{k+1})$ and so $x_{k+1}\in\IB(\bar x, \bar c\bar \delta)$ by Lemma \ref{lemma1-LC}(ii).  In particular, one can conclude by \eqref{remarklm0} and Remark \ref{remark-partial02}(a) that
Algorithm {\rm\ref{GDA}} employing the Armijo step sizes with initial point $x_0\in\IB(\bar x,\bar\delta)$ is well defined, and the generated  sequence $\{x_k\}$
satisfies 
\begin{equation}\label{F-G-B}
\lim_{k\rightarrow+\infty}f(x_k)\ge f(\bar x).
\end{equation}
\end{remark}

For the remainder of this section,  we  always  assume, without loss of generality, that Algorithm \ref{GDA} does not terminate in finite steps.   This particularly implies that, for each $k\in\IN$, $f$ is differentiable at $x_k$ and $t_k$ exists to satisfy \eqref{GDA-2}.
Now, we are ready to show the first main result of this subsection.

\begin{theorem}\label{Local-Convergence}
Let $\bar x\in M$ be such that assumption \eqref{assumption-a} holds and let $f$ be quasi-convex around $\bar x$.
Then, for any $\delta>0$, there exist $\bar \delta>0$ and $\bar c\ge3$ satisfying $\bar c\bar \delta<\delta$ such that,
for any sequence $\{x_k\}$  generated by Algorithm {\rm\ref{GDA}} with initial point $x_0\in\IB(\bar x,\bar\delta)$, if it satisfies \eqref{F-G-B} (e.g., $\bar x$ is a local minimizer of $f$), then one has the following assertions:

 {\rm (i)}  The sequence $\{x_k\}$ stays in $\IB(\bar x,\bar c\bar \delta)$ and converges to a point $x^*$ in ${\rm \cal D}(f)$.

 {\rm (ii)}  If it is additionally assumed  that   $\{t_k\}$  has  a positive lower bound or  that $\{t_k\}$  satisfies the Armijo step sizes, then $x^*$ is a critical point of $f$.
\end{theorem}

\begin{proof} 
By assumption, Lemma \ref{lemma1-LC}(ii) is applicable. For any $\delta>0$, let $\bar \delta,\bar c>0$ be given as in Lemma \ref{lemma1-LC}(ii) and let $\{x_k\}$  be a sequence generated by Algorithm {\rm\ref{GDA}} with initial point $x_0\in\IB(\bar x,\bar\delta)$ which satisfies \eqref{F-G-B}. Noting that $f$ is quasi-convex around $\bar x$, one inductively sees that $\{x_k\}\subset\IB(\bar x,\bar c\bar \delta)$.
Thus, the first conclusion of assertion (i) is shown, and the sequence $\{x_k\}$ has at least a cluster point $x^*$. Letting $L_{\bar\delta}:=\{x\in \IB(\bar c,\bar x\bar\delta): f(x)\le \inf_{k\in\IN}f(x_k)\}$, one sees 
$x^*\in L_{\bar\delta}$ since $\{f(x_k)\}$ is decreasing and $f$ is continuous on $\IB(\bar x,{\bar c\bar\delta})$ (choose a smaller one if necessary). Then, \eqref{lemma1-LC00c} holds for each $z\in L_{\bar \delta}$. Thanks to $\sum_{k=1}^{\infty} t_k\|\nabla f(x_k)\|^2<+\infty$ by \eqref{Full-p1}, we get that $\{x_k\}$ is quasi-Fej\'{e}r convergent to $L_{\bar\delta}$. Hence, recalling $x^*\in L_{\bar\delta}$, we conclude by Lemma \ref{fejer2} that $\lim_{k\rightarrow\infty}x_k=x^*$. Thus, the second conclusion of assertion (i) is seen to hold. 

Assertion (ii) is a direct consequence of assertion (i) and Remark \ref{remark-partial02}(b) (note that one can choose $\bar \delta,\bar c>0$ such that $\nabla f$ is continuous on $\IB(\bar x,\bar c\bar \delta)$ if necessary). This completes the proof. 
\end{proof}

To study the linear convergence property,  we   introduce in the following definition the
notion of the local 
weak sharp minimizer of order $q$ ($q\ge 1$) for problem \eqref{P-1.1}, which is a direct extension of the corresponding one in the linear spaces to Riemannian manifolds; see, e.g., \cite{Burke2002,Studniaski1999,Ward1994}. In particular, the notion of the local 
weak sharp minimizer of order $1$ coincides with the local 
weak sharp minimizer introduced in \cite{LiMWY2011} by Li et al., where some complete characterizations of which were developed on Riemannian manifolds.

\begin{definition}\label{WSM}
A point $\bar x\in \mathcal{D}(f)$ is said to be
a local
weak sharp minimizer of order $q\ge 1$ for problem \eqref{P-1.1} if there exist $\delta,\alpha>0$ such that 
$$
\alpha{\rm d}_{\bar S}^q(x)\le f(x)-f(\bar x)\quad\mbox{for any }x\in\IB(\bar x,\delta),
$$
where   $\bar S:=\{x\in M:f(x)=f(\bar x)\}$.
\end{definition}

Our second main result in this subsection is on the linear convergence property of Algorithm {\rm\ref{GDA}}.

\begin{theorem}\label{Local-Linear}
Suppose that  $\beta\in(\frac12,1)$ and   that $\inf_{k\ge 0} \{t_k\}>0$.  Let
$\bar x\in M$ be such that assumption \eqref{assumption-a} holds,  
\begin{equation}\label{assumptipn-b01}
\mbox{$f$ is convex around $\bar x$, and $\bar x$ is a local  weak  sharp minimizer of order $2$ for  \eqref{P-1.1}}
\end{equation}
%
%
 Then, there exists $\bar \delta>0$ such that any sequence $\{x_k\}$
generated by Algorithm {\rm\ref{GDA}} with   initial point $x_0\in\IB(\bar x,\bar\delta)$ 
 converges linearly to  a local minimizer $x^*$ of $f$, namely there exist $\mu>0$ and $\rho\in(0,1)$ such that
\begin{equation}\label{L-T}
{\rm d}(x_k,x^*)\le \mu\rho^k\quad\mbox{ for each }k\in\IN.
\end{equation}
\end{theorem}

\begin{proof}
By assumption, 
Theorem \ref{Local-Convergence} is applicable to getting that for any $\delta>0$ there exist $\bar\delta>0$, $\bar c>3$ (satisfying $\bar c\bar\delta\le \delta$) such that the sequence $\{x_k\}$ generated by Algorithm {\rm\ref{GDA}} with initial point $x_0\in\IB(\bar x,\bar\delta)$ satisfies
\begin{equation}\label{TL-1}
\{x_k\}\subset\IB(\bar x,\bar c\bar\delta)
\end{equation}
and converges to a critical point $x^*\in\IB(\bar x,\bar c\bar\delta)$. Without loss of generality, we assume further that $2\bar c\bar\delta<r_{\rm cvx}({\bar x})$, and assume by assumption 
that
\begin{equation}\label{local-c}
\mbox{$\IB(\bar x,2\bar c\bar\delta)\subset\mathcal{D}^1(f)$ and $f$ is 
convex on $\IB(\bar x,2\bar c\bar\delta)$}
\end{equation}
(and so $x^*$ is a local minimizer of $f$) and there exists $\alpha>0$ such  that
\begin{equation}\label{local-m}
\alpha{\rm d}_{\bar S}^2(x)\le f(x)-f(\bar x)\quad\mbox{for any }x\in\IB(\bar x,\bar c\bar\delta),
\end{equation}
where ${\bar S}:=\{x\in M:f(x)=f(\bar x)\}$. Below, we show that
\begin{equation}\label{local-m-L}
\|\nabla f(x)\|^2
\ge \alpha( f(x)-f(\bar x))\quad\mbox{for any }x\in \IB(\bar x,\bar c\bar\delta).
\end{equation}
To proceed, let $x\in\IB(\bar x,\bar c\bar\delta)$ and $z\in P_{{\bar S}}(x)$ (then $z\in \IB(\bar x,2\bar c\bar\delta)$ as ${\rm d} (z,\bar x)\le {\rm d} (z,x)+{\rm d} (x,\bar x)\le 2\bar c\bar\delta)$). Recalling \eqref{local-c}, we see from Lemma \ref{QC-F}(ii) that
$$f(x)-f(\bar x)=f(x)-f(z)\le \langle\nabla f(x),-{\exp}_x^{-1}z\rangle\le \|\nabla f(x)\|{\rm d}(x,z)=\|\nabla f(x)\|{\rm d}_{{\bar S}}(x).$$
Thus, \eqref{local-m-L} holds thanks to \eqref{local-m}.

Now, letting $\b{t}:=\inf_{k\in\IN}t_k>0$ ($\b t>0$ is because of assumption (b)), we get that
$$
\begin{array}{lll}
f(x_{k+1})-f(\bar x)&\le f(x_{k})-f(\bar x)-\beta t_k\|\nabla f(x)\|^2\\
&\le(1-\alpha\beta t_k)( f(x_k)-f(\bar x))\\
&\le(1-\alpha\beta\b{t})( f(x_k)-f(\bar x)),
\end{array}
$$
where the first inequality is from \eqref{GDA-2} and the second inequality is because of \eqref{TL-1} and \eqref{local-m-L}.
Then,  there holds
$$
\begin{array}{lll}
f(x_{k})-f(\bar x)\le(1-\alpha\beta\b{t})^{k}( f(x_0)-f(\bar x))\quad\mbox{for each }k\in\IN.
\end{array}
$$
This, together with \eqref{TL-1} and \eqref{local-m}, implies that
\begin{equation}\label{local-m-L2}
\begin{array}{lll}
{\rm d}_{{\bar S}}(x_k)\le\sqrt{\alpha^{-1}( f(x_0)-f(\bar x))}(1-\alpha\beta\b{t})^{\frac{k}{2}}\quad\mbox{for each }k\in\IN
\end{array}
\end{equation}
(note that the above analysis works for all $\beta\in (0,1)$).

On the other hand, by assumption $\beta\in(\frac12,1)$ and the fact $\lim_{t\rightarrow0}\hbar(t)=1$, one can choose $\bar \delta,\bar c$ in the beginning of the proof such that they additionally satisfy $\frac{1}{2\beta}-\hbar(2\sqrt{|\kappa|}\bar c\bar\delta)\le0$.
Then, there holds from Lemma \ref{lemma1-LC}(i) that
\begin{equation}\label{L-L3}
\begin{array}{lll}
{\rm d}(x_{k+1},z)\le{\rm d}(x_k,z)
\quad\mbox{for any $z\in L\cap\IB(\bar x,2\bar\delta)$ and $k\in\IN$},
\end{array}
\end{equation}
where $L:=\{x\in M:f(x)\le \inf_{k\in\IN}f(x_k)\}$. 
This in particular implies that $\{x_k\}\subset\IB(\bar x,\bar \delta)$ as $\bar x\in L$. 
Taking $\bar x_k\in P_{{\bar S}}(x_k)$, we see that $\bar x_k\in {{\bar S}}\cap\IB(\bar x,2\bar\delta)$ (as ${\rm d} (\bar x_k,\bar x)\le {\rm d} (\bar x_k, x_k)+{\rm d} (x_k,\bar x)\le 2\bar\delta)$). Let $l,k\in\IN$ with $l>k$. Then, we get that
$${\rm d}(x_l,x_k)\le{\rm d}(x_l,\bar x_k)+{\rm d}(\bar x_k,x_k)\le 2{\rm d}_{{\bar S}}(x_{k})\le2\sqrt{\alpha^{-1}( f(x_K)-f(\bar x))}(1-\alpha\beta\b{t})^{\frac{k-K}{2}},$$
where the second inequality is because of \eqref{L-L3} and the third one is from \eqref{local-m-L2}. Letting $l$ go to infinity, there holds that
$$
\begin{array}{ll}
{\rm d}(x_k,x^*)\le2\sqrt{\alpha^{-1}( f(x_K)-f(\bar x))}(1-\alpha\beta\b{t})^{\frac{k-K}{2}}\quad\mbox{for each }k\in\IN,
\end{array}
$$
which yields \eqref{L-T} with $\mu:=2\sqrt{\alpha^{-1}( f(x_0)-f(\bar x))}$ and $\rho:=\sqrt{1-\alpha\beta\b{t}}$. The proof is complete.
\end{proof}



\begin{remark}\label{remark-L-LSM}
If the local minimizer $\bar x$ in Theorem \ref{Local-Linear} is  also isolated (namely,   $f(\cdot)>f(\bar x)$ on $U\setminus\{\bar x\}$ for some neighbourhood $U$ of $\bar x$), then the parameter $\beta$ can be relaxed to be in $(0,1)$. 
In fact, under the same assumptions made in Theorem \ref{Local-Linear} but relaxing $\beta\in (0,1)$, if $\bar x$ in Theorem \ref{Local-Linear} is an isolated local minimizer, then one can choose $\delta>0$ small enough at the beginning of the proof of Theorem \ref{Local-Linear} such that ${\rm d}_{\bar S}(x):={\rm d}(x, \bar x)$ for any $x\in \IB(\bar x, \delta)$, and $\{x_k\}$ 
satisfies \eqref{TL-1} (with $\bar c\bar\delta\le \delta$) and  
\eqref{local-m-L2} (as we have noted in the proof). Thus, the 
result is immediate from \eqref{local-m-L2} and the fact that ${\rm d}_{\bar S}(x_k):={\rm d}(x_k, \bar x)$ for each $k\in \IN$ (noting \eqref{TL-1}).

%
\end{remark}


The following lemma provides a sufficient condition for the step size sequence $\{t_k\}$ generated by the Armijo step sizes to have a positive lower bound.

\begin{lemma}\label{Bt-Lip}
Let $\{x_k\}$ be a sequence generated by Algorithm {\rm\ref{GDA}} employing 
the Armijo step sizes. Suppose that $\{x_k\}$
converges to a point $x^*\in \mathcal{D}^1(f)$ and  
  $\nabla f$ is Lipschitz continuous around $x^*$. Then, the step size sequence $\{t_k\}$ has a positive lower bound:
$\inf_{k\in\IN}t_k>0$. 
\end{lemma}

\begin{proof}
By assumption, there exist $\delta, L>0$ (with $\delta\le r_{\rm cvx}({x^*})$) such that $\IB(x^*,3\delta)\subset\mathcal{D}^1(f)$ and
\begin{equation}\label{Lip-1-G}
\|\nabla f(x)-P_{x,y}\nabla f(y)\|\le L{\rm d}(x,y)\quad\mbox{for any }x,y\in\IB(x^*,3\delta).
\end{equation}
Noting that $\lim_{k\rightarrow+\infty}x_k=x^*$, there is $K\in\IN$ such that
$$
x_k\in\IB(x^*,\delta)\quad\mbox{for each }k\ge K.
$$
Fix $k\ge K$, and assume that $t_k\le\frac12$. Then, by \eqref{GDA-2-A}, we see that
\begin{equation}\label{Lip-3-Bt}
f(\gamma_{k}(2t_{k}))- f(x_{k})\ge -2\beta t_{k}\|\nabla f(x_{k})\|^2.
\end{equation}
Moreover, from the mean value theorem (as $\IB(x^*,3\delta)\subset\mathcal{D}^1(f)$), we get that
\begin{equation}\label{Full-p5}
\begin{array}{lll}
&f(\gamma_{k}(2t_k))-f(x_{k})=\left\langle\nabla f\left(\gamma_{k}(2\bar t_k)\right),-2t_kP_{\gamma_k,\gamma_{k}(2\bar t_k),x_k}\nabla f(x_{k})\right\rangle\\
&=-2t_k\left\langle P_{\gamma_k,x_k,\gamma_{k}(2\bar t_k)}\nabla f\left(\gamma_{k}(2\bar t_k)\right)-\nabla f(x_{k}),\nabla f(x_{k})\right\rangle-2t_k\left\langle \nabla f(x_{k}),\nabla f(x_{k})\right\rangle\\
&=2t_k\|P_{\gamma_k,x_k,\gamma_{k}(2\bar t_k)}\nabla f\left(\gamma_{k}(2\bar t_k)\right)-\nabla f(x_{k})\|\cdot\|\nabla f(x_{k})\|-2t_k\|\nabla f(x_{k})\|^2,
\end{array}
\end{equation}
where $\bar t_k\in(0,t_k)$ and $\gamma_k$ is the geodesic determined by \eqref{GDA-1}.
Noting that $\IB(x^*,\delta)$ is strongly convex, one sees that $\gamma_k({[0,t_k]})$ is the unique minimal geodesic joining $x_k$ to $x_{k+1}$; hence
$${\rm d}(x_{k+1},\gamma_{k}(2\bar t_k))={\rm d}(\gamma_{k}(t_k),\gamma_{k}(2\bar t_k))=|t_k-2\bar t_k|\|\nabla f(x_k)\|\le t_k\|\nabla f(x_k)\|={\rm d}(x_k,x_{k+1}).$$
(noting that the equality of \eqref{D-tpp} holds). It follows that
$$
{\rm d}(x^*,\gamma_{k}(2\bar t_k))\le {\rm d}(x^*,x_{k+1})+{\rm d}(x_{k+1},\gamma_{k}(2\bar t_k))\le 3\delta.$$
Thus, we get from \eqref{Lip-1-G} that
$$
\|P_{\gamma_k,x_k,\gamma_{k}(2\bar t_k)}\nabla f\left(\gamma_{k}(2\bar t_k)\right)-\nabla f(x_{k})\|\le2Lt_k\|\nabla f(x_{k})\|
$$
noting that ${\rm d}(x_k,\gamma_{k}(2\bar t_k))\le 2t_k\|\nabla f(x_{k})\|$.
This, together with \eqref{Full-p5}, 
implies that
$$f(\gamma_{k}(2t_k))-f(x_{k})\le2t_k(2L t_k-1)\|\nabla f(x_{k})\|^2.$$
Combining this and \eqref{Lip-3-Bt}, we conclude that $t_k\ge \frac{1-\beta }{2L}$ (if $k\ge K$ and $t_k\le \frac12$).
Thus,
$$\inf_{k\in\IN}t_k=\min\left\{\frac12,t_0,\dots,t_{K-1},\frac{1-\beta }{2L}\right\}>0$$
as desired. 
\end{proof}

In   spirit of the notion of  a nondegenerate critical point (in the sense that $\nabla ^2f(\bar x)$  is positive definite; see \cite[Definition 3.1]{Yang2007}), we say that  a  point $\bar x\in\mathcal{D}^2(f)$  is a quasi-nondegenerate critical point of $f$
if 

{\rm (1)} $f$ is convex around  $\bar x$, and $\nabla f$ is Lipschitz continuous around  $\bar x$;

 {\rm (2)} $\bar x$ is a local weak  sharp minimizer of order $2$ for problem \eqref{P-1.1}.\\
By definition it is clear that  a nondegenerate critical point is also a quasi-nondegenerate critical point.
%
We have the following result regarding the linear convergence of Algorithm {\rm\ref{GDA}} employing the Armijo step sizes around a quasi-nondegenerate critical point of $f$.

\begin{corollary}\label{C-LL-Q}
Let $\bar x$ be a quasi-nondegenerate point of $f$ and let $\beta\in(\frac12,1)$ or $\bar x$   be  isolated. 
Then, 
there exists $\delta>0$ such that any sequence $\{x_k\}$
generated by Algorithm {\rm\ref{GDA}} employing the Armijo step sizes with  initial point $x_0\in\IB(\bar x,\delta)$ converges linearly to a local minimizer of $f$.
\end{corollary}

\begin{proof}
By assumption, we see that assumptions 
\eqref{assumption-a} and \eqref{assumptipn-b01} hold. Then,  Theorem \ref{Local-Convergence}(i) is applicable and so there exists $\delta>0$ such that any sequence $\{x_k\}$ generated by Algorithm {\rm\ref{GDA}} with initial point $x_0\in\IB(\bar x,\delta)$ converges to a point $x^*\in \IB(\bar x,\delta)$. 
Noting that $\nabla f$ is Lipschitz continuous around $\bar x$ (as $\bar x$ is a quasi-nondegenerate point of $f$), we get that $\nabla f$ is Lipschitz continuous around $x^*$ (choose a smaller $\delta$ if necessary).
Thus, one applies Lemma \ref{Bt-Lip} to getting $\inf \{t_k\}>0$.
Hence, Theorem \ref{Local-Linear} is applicable to completing the proof.
\end{proof}


\subsection{Global convergence}

The following theorem regards the global convergence and the linear convergence of Algorithm \ref{GDA}. We emphasize that the convergence result as well as the linear convergence rate of Algorithm \ref{GDA} is independent of the curvatures of $M$.
In particular, in the case when the algorithm employs the Armijo step sizes, 
assertion (ii) extends the corresponding results in \cite[Theorem 4.1]{Yang2007}, which was proven under the assumption that $\{x_k\}$ converges to a nondegenerate point $\bar x$ (noting that this clearly implies that \eqref{assumptipn-b01} holds and $\bar x$ is isolated, and that $\inf_{k\ge 0} \{t_k\}>0$
by Lemma \ref{Bt-Lip}). 

\begin{theorem}\label{full-1}
Suppose that the sequence $\{x_k\}$ generated by Algorithm {\rm\ref{GDA}} has a cluster point $\bar x\in\mathcal{D}(f)$ 
such that assumption \eqref{assumption-a} holds. Then, the following assertions hold:

{\rm (i)} If 
$f$ is quasi-convex around $\bar x$, 
then $\{x_k\}$ converges to $\bar x$.

{\rm (ii)} If $\inf_{k\ge 0} \{t_k\}>0$ and 
assumption \eqref{assumptipn-b01} holds, 
then
$\{x_k\}$ converges linearly to $\bar x$ provided that either  $\beta\in(\frac12,1)$ or $\bar x$   is  isolated. 

\end{theorem}

\begin{proof}
Suppose that $f$ is quasi-convex around $\bar x$. Noting that \eqref{F-G-B} is naturally satisfied as $\{f(x_k)\}$ is non-increasing monotone and $\bar x$ is a cluster point, we get from Theorem \ref{Local-Convergence}(i) 
that there exists $\delta>0$ such that any sequence generated by Algorithm {\rm\ref{GDA}} with initial point in $\IB(\bar x,\delta)$ is convergent. Now $\bar x$ is a cluster point, so there exists some $k_0\in \IN$ such that $x_{k_0}\in \IB(\bar x,\delta)$. Thus, $\{x_k\}$ converges to some point, which in fact equals to $\bar x$ and assertion (i) holds.

With a similar argument that we did for assertion (i), but using Theorem \ref{Local-Linear} (and Remark \ref{remark-L-LSM}) instead
of Theorem \ref{Local-Convergence}(i), one sees that assertions (ii) holds. 
The proof is complete.
\end{proof}

%
%
%
%
The following lemma provides some sufficient assumptions ensuring the boundedness of the sequence $\{x_k\}$ generated by Algorithm {\rm\ref{GDA}} (and so the existence of a cluster point). Let $L_{f}(c)$ denote the sub-level set of $f$ associated with constant $c\in\IR$, that is, $L_{f}(c):=\{x\in M:f(x)\le c\}$. In particular, let $L^0_f:=L_{f}(f(x_0))$ for simplicity.

\begin{lemma}\label{lemma-BX}
Let $\{x_k\}$ be a sequence generated by Algorithm {\rm\ref{GDA}} with initial point $x_0\in \mathcal{D}^1(f)$. Then, $\{x_k\}$ is bounded 
provided one of the assumptions {\rm (a)} and {\rm (b)} holds:

{\rm (a)} $L^0_f$ is bounded.

{\rm (b)}  $L^0_f$ is totally convex with  its curvatures  being  bounded from below and $f$ is quasi-convex  on   $L^0_f$ (e.g., $f$ is quasi-convex  on $M$ and $M$ is of lower bounded curvatures).
\end{lemma}


\begin{proof}
Note that $\{x_k\}\subseteq L^0_f$ as $\{f(x_k)\}$ is non-increasing monotone. Then, $\{x_k\}$ is clear bounded under assumption {\rm (a)} .

Now, suppose that assumption {\rm (b)}  holds.
Without loss of generality, we assume that the curvatures of $L^0_f$
are bounded from below by $\kappa=-1$. To proceed, let $z\in L:=\{x\in M:f(x)\le \inf_{k\in\IN}f(x_k)\}$.
Then, we see that 
$
\{z\}\cup\{x_k\}\subseteq L^0_f 
$ 
because $\{f(x_k)\}$ is non-increasing. Note that $f(x_k)\ge f(z)$ for each $k\in\IN$. Then, by assumption, Lemma \ref{basic-QC}(ii) is applicable on $Q_f:=L^0_f$ (with $x_{k}$, $\gamma_{k}$ and $t_{k}$ in place of $x$, $\gamma$ and $t$) to getting that for each $k\in\IN$,
\begin{equation}\label{Full-11}
\cosh\left({\rm d}(x_{k+1},z)\right)\leq
\cosh\left({\rm d}(x_k,z)\right)\left(1+\frac{1}{2}t_k\|\nabla f(x_k)\|\sinh(t_k\|\nabla f(x_k)\|)\right).
\end{equation}
Note further that
$$
\sum_{k\in\IN}t_k\|\nabla f(x_k)\|\sinh(t_k\|\nabla f(x_k)\|)<+\infty
$$
as $
\sum_{k\in\IN}t_k^2\|\nabla f(x_k)\|^2<+\infty
$ 
(by \eqref{Full-p1} and $\sup \{t_k\}\le R$) 
and $\lim_{t\rightarrow0}\frac{\sinh t}{t}=1$.
In view of \eqref{Full-11}, Lemma \ref{lemma3} is applicable (with $\{\frac{1}{2}t_k\|\nabla f(x_k)\|\sinh(t_k\|\nabla f(x_k)\|\}$ and $\{\cosh\left({\rm d}(x_k,z)\right)\}$ in place of $\{b_k\}$ and $\{a_k\}$), and we get that $\{\cosh\left({\rm d}(x_k,z)\right)\}$ is bounded, and so is $\{x_k\}$ as disired. The  proof is complete. 
\end{proof}

The following corollary is immediate from Theorem \ref{full-1}  and Lemma \ref{lemma-BX}. Particularly, in view of Remark \ref{remark-partial02}, the global convergence result (assertion (i)) under assumption {\rm (b)}  extends the corresponding one in \cite[Theorem 3.1]{Papa2008} which was established on the Riemannian manifold of  
nonnegative curvatures for the case when $f$ is $C^1$ and quasi-convex on $M$. As for assertion (ii), as far as we know, it is new in Riemnnain manifold settings.

\begin{corollary}\label{full-AG}
Suppose that one of assumptions {\rm (a)}  and {\rm (b)}  in Lemma \ref{lemma-BX} holds.
Then, any sequence $\{x_k\}$ generated by Algorithm {\rm\ref{GDA}} with initial point $x_0\in \mathcal{D}^1(f)$ has at least a cluster point $\bar x$; furthermore, if $\bar x$ satisfies \eqref{assumption-a}, then assertions {\rm (i)} and {\rm (ii)} in Theorem \ref{full-1} hold.
\end{corollary}

\section{Applications to find the Riemannian $L^p$ centers of mass}

Let $p\in[1,+\infty)$ and let $N$ be a positive integer such that $N\ge 2$.
Let $\{y_i:1\le i\le N\}\subset M$ (which is always denoted by $\{y_i\}$ for short in what follows) be a data set
and $\{w_i\} \subseteq (0,1)$  be the weights satisfying  $\sum_{i=1}^N w_i=1$. 
In the present section, we shall apply the gradient algorithm proposed in the previous section 
to compute the Riemannian $L^p$  centers of mass of the data set $\{y_i\}$,
  which are defined as  solutions of the following optimization problem:
\begin{equation}\label{OP-LP-def}
\min_{x\in M}f_p(x),
\end{equation}
where the function $f_p:M\rightarrow\IR$ is defined by
\begin{equation}\label{LP-def}
f_p(x):=
\frac{1}{p}\sum_{i=1}^{N}w_i{\rm d}^p(x,y_i)
\quad\mbox{for any }x\in M,
\end{equation}
(see, e.g., \cite[Definition 2.5]{Afsari2013}). 
From now on, for convenience, we set
$$
D:=\bigcap_{i=1}^N\mathbb{U}(y_i,r_{\rm inj}(y_i)).
$$
Let $D_0\subseteq D$ be an open nonempty subset. Now consider the following optimization problem:
\begin{equation}\label{OP-CM}
\min_{x\in M}\;(f_p +\delta_{D_0})(x),
\end{equation}
where $\delta_{D_0}$ is the indicator function defined by $\delta_{D_0}(x)=0$ if $x\in D_0$ and
$\delta_{D_0}(x)=+\infty$ otherwise. 

The following remark shows some properties of the function $f_p$ defined in \eqref{LP-def}.
For convenience, we set $I:=\{1,2,\dots,N\}$.

\begin{remark}\label{remark-LM-ProA}
The function $f_p+\delta_{D}$ is $C^1$ on $D$ if $p\in(1,+\infty)$; furthermore, it is $C^2$ on $D$ if   $p\in[2,+\infty)$ and on $D\setminus\{y_i\}$ if $p\in[1,2)$; see, e.g., \cite[p. 108-110]{Sakai1996}.
Moreover, if $f_p+\delta_{D}$ is differentiable at $x\in D$, then
\begin{equation}\label{derivative}
\nabla (f_p+\delta_{D})(x)=\nabla f_p(x)=-\sum_{i\in I_x}w_i{\rm
d}^{p-2}(x,y_i)\exp^{-1}_xy_i,
\end{equation}
where $I_x:=\{i\in I:x\neq y_i\}$;
see, e.g., \cite{Afsari2010}.
%
\end{remark}

Below, we recall some results about the Riemannian centers of mass in the literature. To proceed, we fix a point $o\in M$ and 
define the function $\varrho_p: (0,+\infty)\rightarrow\overline{\IR}$ by
\begin{equation}\label{Delta-p}
\varrho_p(r):=\left\{\begin{array}{ll}\frac{1}{2}\min\{r_{\rm inj}(\IB(o,2r)),
\frac{\pi}{2\sqrt{\Delta_{\IB(o,2r)}}}\},&\quad\mbox{if $1\le p<2$};\\
\frac{1}{2}\min\{r_{\rm inj}(\IB(o,2r)),\frac{\pi}{\sqrt{\Delta_{\IB(o,2r)}}}\},&\quad\mbox{if $2\le p<+\infty$},\end{array}\right.\quad\mbox{for each }r\in(0,+\infty),
\end{equation}
where 
$\Delta_{\IB(o,2r)}$ is an upper bound of the sectional curvatures of ${\IB(o,2r)}$ (with the convention that $\frac{1}{\sqrt{\Delta}}=+\infty$ for $\Delta\le 0$).
Then, $\varrho_p(\cdot)$ is non-increasing monotonically on $(0,+\infty)$. For the remainder, let $\rho\in(0,+\infty]$ be such that 
\begin{equation}\label{jbjsh01}
 \rho\le\varrho_p(\rho)  \quad \mbox{and}\quad \{y_i\}\subset \mathbb{U}(o,\rho).
\end{equation}
In what follows, we need the following fact which can be found in \cite[Theorem 29]{Petersen2006}.

\begin{lemma}\label{PPR}
Let $r>0$ be such that $r\le \frac{1}{2}\min\{r_{\rm inj}(\IB(o,r)),\frac{\pi}{\sqrt{\Delta_{\IB(o,r)}}}\}$.
Then, $\mathbb{U}(o,r)$ is strongly convex. 
\end{lemma}


\begin{lemma}\label{weaklocal}
Assume \eqref{jbjsh01} and let $z\in\partial \IB(o,\rho)$. Then, the following assertions hold:

{\rm (i)} 
   $\IB(o,\rho)\subset D$,  and $\mathbb{U}(o,\rho)$ is strongly convex (so $\IB(o,\rho)$ is weakly convex).

{\rm (ii)}  If $ y\in \mathbb{U}(o,\rho)$ and $\gamma\in \Gamma_{yz}$   is  minimal, then $\gamma([0,1))\subseteq \mathbb{U}(o,\rho)$.

{\rm (iii)} There exists $\bar s>0$ such that
\begin{equation}\label{gradinp0}
 \exp_{z}(-s\nabla f_p(z))\in\mathbb{U}(o,\rho) \quad\mbox{for any } s\in (0,\bar s].
\end{equation}
\end{lemma}

\begin{proof} (i) The inclusion  $\IB(o,\rho)\subset D$ is clear because  for each $i\in I$, 
$${\rm d}(x,y_i)<2\rho\le2\varrho_p(\rho)\le r_{\rm inj}(\IB(o,2\rho)\le r_{\rm inj}y_i \quad\mbox{for any }x\in \IB(o,\rho),$$
where  the third  inequality  is true by the definition of $\varrho_p$ (see \eqref{Delta-p}), while the others hold by   assumption \eqref{jbjsh01}.
Furthermore, the strong convexity of $\mathbb{U}(o,\rho)$ is from Lemma \ref{PPR} and assumption \eqref{jbjsh01}.

(ii)  Let $ y\in \mathbb{U}(o,\rho)$. 
 To show (ii), we verify below that  $y$ is a  weak pole of $\IB(o,\rho)$ in the sense that, for each $x\in \IB(o,\rho)$, the minimal geodesic of
$M$ joining $y$ to $x$ is unique and lies in $\IB(o,\rho)$.  Granting this, the conclusion holds by \cite[Proposition 4.3]{LiLi2009} (noting that weakly convex set is locally convex). To proceed, recalling that $\IB(o,\rho)$ is weakly convex, one can choose
a minimal geodesic $\gamma$  joining $y$ to $x$ such that $\gamma\subset \IB(o,\rho)$.
Let 
$w$ be the midpoint of $\gamma$. Note that the length
$l(\gamma)<2\rho$. One sees that  $y,x\in \mathbb{U}(w,\rho)$, and $\mathbb{U}(w,\rho)\subset\mathbb{U}(o,2\rho)$. By assumption \eqref{jbjsh01},
  there holds that
$$\rho\le \frac{1}{2}\min\{r_{\rm inj}(\IB(o,2\rho)),\frac{\pi}{\sqrt{\Delta_{\IB(o,2\rho)}}}\}\le \frac{1}{2}\min\{r_{\rm inj}(\IB(w,\rho)),
\frac{\pi}{\sqrt{\Delta_{\IB(w,\rho)}}}\}.$$
Thus,  Lemma \ref{PPR} is applicable to concluding that   $\IB(w,\rho)$ is strongly convex and so $\gamma$ is the unique minimal geodesic joining $y$ to $x$ (noting that $x,y\in \IB(w,\rho)$). This shows that  $y$ is a weak pole of $\IB(o,\rho)$ as desired, and assertion (ii) is established.

(iii) 
Fix  $i\in I$, and write $V_i:=\frac{\exp^{-1}_{z}y_i}{\|\exp^{-1}_{z}y_i\|}$. The geodesic $[0,1]\ni t\mapsto \exp_{z}(t\|\exp^{-1}_{z}y_i\|V_i)$
is the minimal geodesic joining $z$ and $y_i$. Applying assertion (ii) just established (to $y_i$ in place of $y$), one checks that
$$
 \exp_{z}(s\|\exp^{-1}_{z}y_i\|V_i)\in \mathbb{U}(o,{\rho})  \quad \mbox{for any } 0<s<1.
$$
Thus, applying \cite[Lemma H.18]{Rici-ta} to $\{V_1,V_2\}$ (with $\mathbb{B}(o,\rho)$ in place of $C$), one can conclude that
  there exits $s_1>0$ such that
$$
\mbox{$\exp_{z}s({\lambda_1V_1+\lambda_2V_2})\in\mathbb{U}(o,{\rho})$ \quad for $0<s\le s_1$,}
$$
  and then, by mathematical  induction, that there exits $\bar s>0$ such that
$$
 \exp_{z}s\sum_{i\in I}\lambda_iV_i\in\mathbb{U}(o,{\rho}) \quad \mbox{for any }0<s\le\bar s,
$$
where each  $\lambda_i:=w_i{\rm d}^{p-1}(z,y_i)$. Taking into account that $-\nabla f_p(z)=\sum_{i\in I}\lambda_iV_i$ by \eqref{derivative}
(noting that $I_z=I$, thanks to assumption \eqref{jbjsh01} and $z\in\partial \IB(o,\rho)$), we conclude that
\eqref{gradinp0} holds. The proof is complete. 
%
%
%
\end{proof}

For the remainder, in view of Lemma \ref{weaklocal}(i), we choose $D_0:=\mathbb{U}(o,\rho)$ for the problem \eqref{OP-CM} unless otherwise specified. Now we are ready to establish the following key proposition.  Recall that $\{y_i\}$ is   colinear if it lies in one geodesic segment. We also need to make use of the following assumption:
%
\begin{equation}\label{ninnonde-asm}
 \min_{x\in {M}} f_p(x) <\min_{  i\in I}f_p(y_i) . 
\end{equation}

\begin{proposition}\label{LM-lemmaAC}
Assume that \eqref{jbjsh01} holds 
and that  $\{y_i\}$ is  not colinear  if  $p=1$.  
Then, $\{y_i\}$ has the unique Riemannian $L^p$ center of mass $\bar x_p$, which lies in $\mathbb{U}(o,\rho)$ and  is the unique critical point of $f_p$ in $\mathbb{B}(o,\rho)$. Furthermore, the following assertions hold:

{\rm (i)}
$\bar x_p$ is a nondegenerate critical point of $f_p$ (and so $\nabla f_p$ is Lipschitz continuous around $\bar x_p$) if  \eqref{ninnonde-asm} is additionally assumed for  $p\in [1,2)$.

{\rm (ii)} $\bar x_p$ is a local weak sharp minimizer of order $2$ for problem \eqref{OP-CM} if  $p\in (1,2)$.

{\rm (iii)} $f_p$ is convex around $\bar x_p$.
\end{proposition}

\begin{proof}
Note by \eqref{gradinp0} that $\nabla f_p$ does not vanish on $\partial \mathbb{B}(o,\rho)$ thanks to assumption \eqref{jbjsh01}.
Thus, by assumption, it follows from \cite[Theorem 2.1 and Remark 2.5]{Afsari2010} that $\{y_i\}$ has the unique Riemannian $L^p$ center of mass $\bar x_p\in\mathbb{U}(o,\rho)$, which is the unique critical point of $f_p$ in $\mathbb{U}(o,\rho)$.

(i)  Note that the conclusion  for $p\in[2,+\infty)$ follows from \cite[Theorem 2.1]{Afsari2010} (applied to
  $\mathbb{U}(o,2\rho)$ in place of $M$). Thus, we assume that
 $p\in[1,2)$ and   that \eqref{ninnonde-asm} holds.  
Then,    $\bar x_p\in\IB(o,\rho)\setminus\{y_i\}$.
 We shall complete the proof by showing that $\nabla^2 f_p(x)$ is positive definite for each $x\in\IB(o,\rho)\setminus\{y_i\}$.
To do this, let $x\in\IB(o,\rho)\setminus\{y_i\}$, and let $  \gamma(\cdot)$ be a unit
speed geodesic with $\gamma(0)=x$. Then, $\gamma([-\epsilon,\epsilon])\subseteq \IB(o,\rho)\setminus\{y_i\}$ for some $\epsilon>0$. It suffices to verify that
\begin{equation}\label{wsm-pd1}
\frac{{\rm d}^2}{{\rm d}t^2}(f_p\circ\gamma)(0)=\frac{{\rm d}^2}{{\rm d}t^2}f_p(\gamma(t))|_{t=0}>0
\end{equation}
To show this, let $i\in I$, and let $\alpha_i$   denote the angle at $x$ between the geodesic $\gamma$ and the unique minimal
geodesic joining $x$ to $y_i$. Note that
\begin{equation}\label{Hessian-P-x}
{\rm d}(x,y_i)<2\rho\le2\varrho_p(\rho)= \min\{r_{\rm inj}(\IB(o,2\rho)),
\frac{\pi}{2\sqrt{\Delta_{\IB(o,2\rho)}}}\}\quad \mbox{for each $i\in I$},
\end{equation}
thanks to  $\rho\le\varrho_p(\rho)$ (by \eqref{jbjsh01}) and   \eqref{Delta-p}. Then, the function ${\rm d}(\gamma(\cdot),y_i)$
is analytic on $(-\epsilon,\epsilon)$, and by    the arguments for proving \cite[(2.3)]{Absil2008} and \cite[p. 153-154]{Sakai1996},
one has that \begin{equation}\label{Hessian-P}  \frac{{\rm d}}{{\rm d}t}{\rm d}(\gamma(t),y_i)|_{t=0}=\cos\alpha_i\quad\mbox{and}\quad
\frac{{\rm d}^2}{{\rm d}t^2}{\rm d}(\gamma(t),y_i)|_{t=0}\ge c_{\Delta_{\IB(o,2\rho)}}({\rm d}(x,y_i))\sin^2 \alpha_i,
\end{equation}
where, for any $l>0$,  $c_{\delta}(l):=\frac{1}{\sqrt{\delta}}\cot{(\sqrt{\delta}l)}$ if $\delta>0$,  $c_{\delta}(l):=\frac{1}{l}$ if $\delta=0$, and  $c_{\delta}(l):=\frac{1}{\sqrt{|\delta|}}\coth{(\sqrt{|\delta|}l)}$ otherwise.
Since,  for any $t\in(-\epsilon, \epsilon)$,  
$$
\frac{{\rm d}^2}{{\rm d}t^2}f_p(\gamma(t)) =\sum_{i\in I}w_i\left((p-1){\rm
d}^{p-2}(x,y_i)\left(\frac{{\rm d}}{{\rm d}t}{\rm d}(\gamma(t),y_i)\right)^2+{\rm
d}^{p-1}(x,y_i)\frac{{\rm d}^2}{{\rm d}t^2}{\rm d}(\gamma(t),y_i)\right),
$$
it follows from \eqref{Hessian-P} 
that
\begin{equation}\label{wsm-pd}
\frac{{\rm d}^2}{{\rm d}t^2}(f_p\circ\gamma)(0)\ge\sum_{i\in I}w_i\left((p-1){\rm
d}^{p-2}(x,y_i)\cos^2\alpha_i+{\rm
d}^{p-1}(x,y_i)c_{\Delta_{\IB(o,2\rho)}}({\rm d}(x,y_i))\sin^2 \alpha_i\right).
\end{equation}
Note by \eqref{Hessian-P-x} that $0<{\rm d}(x,y_i)<\frac{\pi}{2\sqrt{\Delta_{\IB(o,2\rho)}}}$,
 and then $c_{\Delta_{\IB(o,2\rho)}}({\rm d}(x,y_i))>0$   by definition.
Thus,  \eqref{wsm-pd1} is clear in the case when $p\in(1,2)$; while, for the case when  
$p=1$, there exists  an index $i_0\in I$ such that $\sin \alpha_{i_0}\neq0$ (as $\{y_i\}$ is not colinear by assumption), and \eqref{wsm-pd1} follows  from \eqref{wsm-pd} as
$$
\frac{{\rm d}^2}{{\rm d}t^2}(f_p\circ\gamma)(0)\ge\sum_{i\in I}w_ic_{\Delta_{\IB(o,2\rho)}}({\rm d}(x,y_i))\sin^2 \alpha_{i}\ge w_{i_0}c_{\Delta_{\IB(o,2\rho)}}({\rm d}(x,y_{i_0}))\sin^2 \alpha_{i_0}>0.
$$
Therefore, \eqref{wsm-pd1} is   valid for any $p\in[1,2)$, completing the proof of assertion (i).

(ii) Assume   $p\in (1,2)$. 
 In light of assertion (i), 
we only need to consider the case when  \eqref{ninnonde-asm} is  not satisfied.
Thus, we may assume that  $\bar x_p=y_{i_0}$ for   some $i_0\in I$ and so  $\nabla f_p(y_{i_0})=0$.
Consider the date set $\{y_i:i\in \tilde I\}$ and the weights $\{\tilde w_i:i\in \tilde I\}$, where  $\tilde I:=I\setminus\{i_0\}$ and
 $\tilde w_i:=\frac{w_i}{1-w_{i_0}}$ for each $i\in \tilde I$.
Then, \eqref{jbjsh01} remains true for the date set $\{y_i:i\in \tilde I\}$. Let $\tilde f_p$ denote
  the corresponding function  defined  by \eqref{LP-def} (with $\{y_i:i\in \tilde I\}$,
   $\{\tilde w_i:i\in \tilde I\}$ in place of $\{y_i:i\in  I\}$, $\{  w_i:i\in  I\}$). Then,
  \begin{equation}\label{wsm-ei}
\tilde f_p(\cdot):=\frac{1}{p}\sum_{i\in \tilde I}\tilde w_i{\rm d}^p(\cdot,y_i)=\frac{1}{1-w_{i_0}}\left(f_p(\cdot) -\frac{w_{i_0}}{p}{\rm d}^p(\cdot,y_{i_0})\right).
  \end{equation}
Hence, $ \nabla \tilde f_p({\bar x_p})=\frac{\nabla f_p({\bar x_p})}{1-w_{i_0}}=0$.
This means that  $\bar x_p$ is also the unique Riemannian $L^p$ center of mass of $\{y_i:i\in \tilde I\}$, and so
 \eqref{ninnonde-asm} holds with  $\tilde f_p$, $ \tilde I$ in place of $f_p$, $I $.  %
%
Thus, by assertion (i), 
one sees that $\bar x_p$ is a nondegenerate critical point of $\tilde f_p$, which in particular implies that   $\bar x_p$ is a local weak  sharp minimizer of order $2$ for problem \eqref{OP-CM} with $\tilde f_p$ in place of $f_p$:   there exist $\delta,\alpha>0$ such that 
$$
\alpha{\rm d}^2(x,\bar x_p)\le \tilde f_p(x)-\tilde f_p(\bar x_p)\quad\mbox{for any }x\in\IB(\bar x,\delta).
$$
Since $\tilde f_p(\bar x_p)= \frac{1}{1-w_{i_0}}f_p(\bar x_p)$ and
$\tilde f_p(\cdot)\le \frac{1}{1-w_{i_0}}f_p(\cdot)$ on $\IB(\bar x,\delta)$ by \eqref{wsm-ei}, it follows that
$$
\alpha({1-w_{i_0}}){\rm d}^2(x,\bar x_p)\le ({1-w_{i_0}})(\tilde f_p(x)-\tilde f_p(\bar x_p))\le  f_p(x)-f_p(\bar x_p)\quad\mbox{for any }x\in\IB(\bar x_p,\delta).
$$%
Therfore $\bar x_p$ is a local weak  sharp minimizer of order $2$ for problem \eqref{OP-CM}, establishing assertion (ii).

(iii) It follows from assertion (i) for $p\in[2,+\infty)$ and from \cite[Theorem 2.1]{Afsari2010}) for $p\in [1,2)$
($f_p$ is  actually convex on $\mathbb{U}(o,\rho)$ in the case when $p\in [1,2)$). The proof is complete.
\end{proof}


Recall from Remark \ref{remark-LM-ProA} that $f_p$ is $C^1$ on $D$ for $p\in(1,+\infty)$ and $C^1$ on $D\setminus\{y_i\}$ for $p=1$. Then, we have
\begin{equation}\label{C1fp}
\mathcal{D}^1(f_p +\delta_{D_0})=\left\{\begin{array}{ll}D_0\setminus\{y_i\}&\quad\mbox{if $p=1$},\\
D_0&\quad\mbox{if $p\in(1,+\infty)$}.\end{array}\right.
\end{equation}
Furthermore, it is clear that
\begin{equation}\label{C1fpc}
\mbox{$\nabla (f_p +\delta_{D_0})$ is continuous on $\mathcal{D}^1(f_p +\delta_{D_0})$.}
\end{equation}


\begin{theorem}\label{THA-full-L-mass}
Assume that \eqref{jbjsh01} holds 
and that  $\{y_i\}$ is not colinear for  $p=1$.
Let $\{x_k\}$ be a sequence generated by Algorithm {\rm\ref{GDA}} for solving problem \eqref{OP-CM} with initial point $x_0\in \mathbb{U}(o,{\rho})$, and suppose that the step size sequence $\{t_k\}$ has a positive lower bound: $\inf\{t_k\}>0$ and $\{x_k\}$ has a cluster point $\bar x_p\in \mathcal{D}^1(f_p +\delta_{D_0})$. 
Then, $\{x_k\}$  converges to $\bar x_p$, which is the unique Riemannian $L^p$ center of mass of $\{y_i\}$; moreover  the convergence rate is at least linear if \eqref{ninnonde-asm} is additionally  assumed for $p=1$.
\end{theorem}

\begin{proof}
By \eqref{C1fpc} and the assumption $\inf\{t_k\}>0$,
Remark \ref{remark-partial02}(b) is applicable and we see $\nabla f_p(\bar x_p)=0$, and so \eqref{assumption-a} holds (with $\bar x_p$ in place of $\bar x$).
Then, $\bar x_p\in \mathbb{U}(o,{\rho})$ and is the unique Riemannian $L^p$ center of mass of $\{y_i\}$. Therefore, \eqref{assumptipn-b01} is satisfied from
Proposition \ref{LM-lemmaAC}(iii) if \eqref{ninnonde-asm} is additionally  assumed for $p=1$. 
Thus, Corollary \ref{full-AG} is applicable (noting $\inf\{t_k\}>0$) to completing the proof.
\end{proof}

\begin{corollary}\label{THA-full-L-massc}
Assume that \eqref{jbjsh01} holds 
and that  $\{y_i\}$ is not colinear for  $p=1$.
Let $x_0\in \mathbb{U}(o,{\rho})$ and suppose for $p=1$ that
\begin{equation}\label{p1n}
f_p(x_0) <\min_{i\in I}f_p(y_i).
\end{equation}
Then,   Algorithm {\rm\ref{GDA}} for solving problem \eqref{OP-CM} employing the Armijo step sizes with initial point $x_0$ is well defined, and  
 the generated sequence   $\{x_k\}$ 
converges to the unique Riemannian $L^p$ center of mass of $\{y_i\}$. Moreover,  the
convergence rate is at least linear if \eqref{ninnonde-asm} is additionally  assumed for $p\in[1,2)$.
\end{corollary}

\begin{proof} By \eqref{C1fp}, one sees that $\mathcal{D}^1(f_p +\delta_{D_0})=\mathcal{D}(f_p +\delta_{D_0})=D_0$ in the case when $p\in (1,+\infty)$; thus
the first conclusion regarding the well definedness of Algorithm {\rm\ref{GDA}} follows
directly from Remark \ref{remark-partial02}(a). 
Below we consider  the case 
when $p=1$. To do this,   
in view of \eqref{p1n}, one applies  Remark \ref{remark-partial02}(a) inductively  to check  that
each generated point $\{x_k\}$ satisfying $\{x_k\}\subseteq L^0_f\subset D_0\setminus\{y_i\}=\mathcal{D}^1(f_p +\delta_{D_0})$ (as $\{f_1(x_k)\}$ is decreasing), and so  Algorithm {\rm\ref{GDA}} employing the Armijo step sizes is well defined, completing
the proof for the first conclusion.

To show the second conclusion regarding the convergence rate, we note first that $L^0_f$ is bounded, that is,   assumption (a) in Lemma \ref{lemma-BX} holds. Thus  Corollary \ref{full-AG} is applicable to getting that $\{x_k\}$ has a cluster point, say $\bar x_p\in \IB(o,\rho)$ (noting that $D_0:=\mathbb{U}(o,\rho)$). As noted before, $\{x_k\}\subseteq L^0_f$; hence    $\bar x_p\notin\{y_i\}$ when $p=1$. Then, one sees from \eqref{C1fp} that
$\bar x_p\in\mathcal{D}^1(f_p +\delta_{D_0})\cup\partial D_0$.
Below we show that
$\bar x_p\in \mathcal{D}^1(f_p +\delta_{D_0})$.
Granting this 
and noting that \eqref{ninnonde-asm} additionally holds for $p\in [1,2)$, we get from Proposition \ref{LM-lemmaAC}(i) that $\nabla f_p$ is Lipschitz continuous around $\bar x_p$. Therefore  the corresponding  step size sequence $\{t_k\}$ employing the Armijo step sizes has a positive lower bound thanks to Lemma \ref{Bt-Lip}, and then Theorem \ref{THA-full-L-mass} is applicable to 
completing the proof.

To proceed, suppose on the contrary that $\bar x_p\in\partial D_0$. By assumption \eqref{jbjsh01}, it follows  from Lemma \ref{weaklocal}(iii) that
\begin{equation}\label{LM-lemmaB1}
\nabla f_p(\bar x_p)\neq0
\end{equation}
and there exists $\bar s>0$ such that
\begin{equation}\label{In-pp}
 \exp_{\bar x_p}[-s{\nabla f_p(\bar x_p)}] \in \mathbb{U}(o,{\rho})  \quad \mbox{for any }0< s\le \bar s.
\end{equation}
Now, we show that there exists $\delta_0>0$ such that
\begin{equation}\label{In-pp-x}
\exp_{x}[-s{\nabla f_p(x)}]\in\mathbb{U}(o,{\rho})  \quad \mbox{for any }x\in\IB(\bar x_p,\delta_0) \mbox{ and }0< s\le  \bar s
\end{equation}
(using a smaller $\bar s$ if necessary). To this end,   set $\bar z:=\exp_{\bar x_p}[-\bar s\nabla f_p(\bar x_p)]$ and then $\bar z\in \mathbb{U}(o,{\rho})$  by \eqref{In-pp}; hence  there is $\bar\varepsilon >0$ such that $\IB(\bar z,\bar\varepsilon)\subset \mathbb{U}(o,{\rho})$. Without loss generality, we may assume
\begin{equation}\label{In-pp-Acvx}
\bar s\|\nabla f_p(\bar x_p)\|+\bar\varepsilon+\bar \delta\le r_{\rm cvx}(\IB(\bar x_p,\bar\delta))
\end{equation}
for some $\bar\delta>0$. Since  the mapping $x\mapsto\exp_x[-\bar s\nabla f_p(x)]$ is continuous on $\mathbb{U}(o,{\rho})$ (as $\nabla f_p(x)$ is continuous on $\mathbb{U}(o,{\rho})$), there exists $\delta_0\in (0,\bar \delta)$ such that
$$
\exp_{x}[-\bar s\nabla f_p(x)]\in \IB(\bar z,\bar \varepsilon)\subset \mathbb{U}(o,{\rho})  \quad\mbox{for any }x\in \IB(\bar x_p,\delta_0).
$$
Let $x\in \IB(\bar x_p,\delta_0)$ and write  $z_x:=\exp_{x}[-\bar s\nabla f_p(x)]$. Then, in view of \eqref{In-pp-Acvx}, we check that $${\rm d}(x,z_x)\le {\rm d}(x,\bar x_p)+{\rm d}(\bar x_p,\bar z)+{\rm d}(\bar z,z_x)\le r_{\rm cvx}(\IB(\bar x_p,\delta_1))\le r_{\rm cvx}(x).$$
Thus, the geodesic $[0,\bar s]\ni s\mapsto \exp_{x}[-s\nabla f_p(x)]$ is the minimal geodesic joining $x$ to $z_x$, and so \eqref{In-pp-x} holds as $\mathbb{U}(o,{\rho})$ is strongly convex.

To proceed, let $\{x_{k_j}\}$ be a subsequence of $\{x_k\}$ converging  to $\bar x_p$. Then, $\lim_{j\to+\infty}t_{k_j}=0$
%
by Remark \ref{remark-partial02}(i). Thus, without loss of generality, we may assume that
\begin{equation}\label{LM-lemmaB2}
x_{k_j}\in \IB(\bar x_p,\delta_0)\quad
\mbox{and}\quad
2t_{k_j}\le \bar s\quad\mbox{for each }j.
\end{equation}
Fix  $j$ and recall that the geodesic $\gamma_{k_j}$ is defined  by  \eqref{GDA-1}. Then,  in view of \eqref{In-pp-x} and \eqref{LM-lemmaB2}, we see that
$$
\gamma_{k_j}(s)=\exp_{x_{k_j}}[-s{\nabla f_p(x_{k_j})}]\in \mathbb{U}(o,{\rho}) \quad\mbox{for each }s\in  [0,2t_{k_j}].
$$
By using the mean value theorem, there is $\bar t_{k_j}\in(0,2t_{k_j})$ such that
$$
\frac{f_p(\gamma_{k_j}(2t_{k_j}))-f_p(x_{k_j})}{-2t_{k_j}}=\langle P_{\gamma_k,x_{k_j},\gamma_{k_j}(\bar t_{k_j})}\nabla f_p(\gamma_{k_j}(\bar t_{k_j})),\nabla f_p(x_{k_j})\rangle.
$$
This, together with   (\eqref{GDA-2-A}), implies that
$$\langle P_{\gamma_{k_j},x_k,\gamma_{k_j}(\bar t_{k_j})}\nabla f_p(\gamma(\bar t_{k_j})),\nabla f_p(x_{k_j})\rangle\le \beta \|\nabla f_p(x_{k_j})\|^2.$$
Passing to the limit as  $j\to \infty$, we arrive  at  $\beta\ge1$ by  \eqref{LM-lemmaB1}, which is a contradiction. Thus, the proof is complete.
\end{proof}

Below, we shall
consider the gradient algorithm for solving problem \eqref{OP-LP-def} employing constant step sizes, which is stated as follows.

\begin{algorithm}\label{GA} Give $x_0\in\mathcal{D}(f)$, $t_0\in(0,+\infty)$ and set $k:=0$.

\noindent {\rm{Step 1}}. If $\nabla f(x_k)=0$ or $x_k\notin\mathcal{D}^1(f)$, 
then stop;
otherwise construct $\gamma_{k}$ as \eqref{GDA-1}.

\noindent {\rm{Step 2}}. Set
$x_{k+1}:=\gamma_{k}(t_{0})$,
replace $k$ by $k+1$ and go to step 1.
\end{algorithm}

Let $x_0\in D_0:=\mathbb{U}(o,\rho)$, and we need
the following assumption:
\begin{equation}\label{gcwdinu0}
L^0_{f_p}\subset D_0,
\end{equation}
where, as done in Section 3, $L^0_{f_p}:=L_{f_p}(f_p(x_0))$ is the sub-level set.  Moreover, we need also the following assumption made for
$p\in [1,2)$:
\begin{equation}\label{indif}
f_p(x_0)<\min_{i\in I}f_p(y_i).
\end{equation}
%
Thus, under assumption \eqref{gcwdinu0}, and assumption \eqref{indif} (only for $p\in [1,2)$),
 $f_p$ is $C^2$ on $L_{f_p}^0$ by Remark \ref{remark-LM-ProA}, and the supremum of all eigenvalues of $\nabla^2 f_p(\cdot)$ on $L_{f_p}^0$, denoted by   
  $\lambda_p(x_0)$,   is bounded (as $L^0_{f_p}$ is compact).


%

\begin{corollary}\label{Coro-full-L-mass-1}
Assume that \eqref{jbjsh01} holds and that  $\{y_i\}$ is   not colinear if  $p=1$. Let  $x_0\in D_0$ be  such that
\eqref{gcwdinu0} holds,  and that  \eqref{indif} holds for $p\in[1,2)$. 
Then, Algorithm {\rm\ref{GA}} for solving problem \eqref{OP-LP-def} with 
$t_0\in \left(0,\frac{2}{\lambda_p(x_0)}\right)$ is well defined and converges linearly to the unique $L^p$ center of mass of $\{y_i\}$. 
\end{corollary}

\begin{proof}
As noted earlier,
$f_p$ is $C^2$ on $L_{f_p}^0$, which particularly implies that $L_{f_p}^0\subset\mathcal{D}^1(f_p +\delta_{D_0})$.
Thus, to show the first assertion, it is sufficient to show  that $x_k\in L_{f_p}^0$ for each $k$.
Clearly, $x_0\in L_{f_p}^0$ 
by the choice of $x_0$. 
To proceed, suppose that $x_j\in L_{f_p}^0$ for some $j\in\IN$. 
Let $\gamma_j:[0,+\infty)\rightarrow M$ be the geodesic defined by \eqref{GDA-1}, and set
 $\bar t:=\sup\{t:\gamma_j(s)\in L_{f_p}^0\mbox{ for any }0\le s\le t\}$.
By  the Taylor   expansion and using the upper bound on the Hessian
of $f_p$ on $L_{f_p}^0$, 
we check that, for each $t\in (0,\bar t]$,
$$
\begin{array}{lll}
f_p(\gamma_j(t))&=f_p(x_j)-t\|\nabla f_p(x_j)\|^2+t^2\int_0^1(1-\tau)\langle\nabla^2f_p(\gamma_j(\tau t))\nabla f_p(x_j),\nabla f_p(x_j)\rangle d\tau\\
&\le f_p(x_j)-t\|\nabla f_p(x_j)\|^2+\frac{t^2\lambda_p(x_0)}{2}\|\nabla f_p(x_j)\|^2.
\end{array}
$$
Noting $f_p(x_j)\le f_p(x_0)$, it follows that
\begin{equation}\label{Taylor-LM}
f_p(\gamma_j(t))\le f_p(x_0)-t(1-\frac{t\lambda_p(x_0)}{2})\|\nabla f_p(x_j)\|^2\quad\mbox{for each }t\in [0,\bar t].
\end{equation}
This 
implies that  $
\bar t\ge \frac{2}{\lambda_p(x_0)}> t_0 $ by definition of $\bar t$ and continuinity of $f_p$. 
Therefore, $x_{j+1}:=\gamma_j(t_0)\in L_{f_p}^0$, and then, by mathematical induction,
$x_k\in L_{f_p}^0$ for each $k\in\IN$ as desired to show. Furthermore, \eqref{Taylor-LM} implies that the generated sequence $\{x_k\}$ by Algorithm \ref{GA} satifies
$$
f(\gamma_k(t_0))\le f(x_k)-t_0\beta\|\nabla f(x_k)\|^2\quad \mbox{for each }k\in\IN,
$$
where $\beta:=1-\frac{t_0\lambda_p(x_0)}{2}\in(0,1)$. This means that $\{x_k\}$ coincides with the sequence generated
by Algorithm {\rm\ref{GDA}} for solving problem \eqref{OP-CM} with initial point $x_0$ and
constant step sizes $\{t_k:=t_0\}$.
Note that $\{x_k\}\subset L_{f_p}^0$ and $L_{f_p}^0\subset\mathcal{D}^1(f_p +\delta_{D_0})$. Then, $\{x_k\}$ has a cluster point in $\mathcal{D}^1(f_p +\delta_{D_0})$ as $L_{f_p}^0$ is clearly blounded.
Furthermore, \eqref{indif} particularly implies \eqref{ninnonde-asm}. Thus, Theorem \ref{THA-full-L-mass} is applicable and
$\{x_k\}$ converges  linearly to the unique Riemannian $L^p$ center of mass of $\{y_i\}$. The proof is complete.
\end{proof}

The following corollary is new in the case when $p\in[1,2)$, and   was proved  in \cite[Theorem 4.1]{Afsari2013}
in the case when   $p\in[2,+\infty)$ under the assumption that  $\{y_i\}\subset\IB(o,\frac13 r_{cx})$ with
$r_{cx}:=\frac{1}{2}\min\{r_{\rm inj} (M),\frac{\pi}{\sqrt{\Delta_M}}\}$), which
 particularly implies the following assumption \eqref{jbjsh01A}  with $r_{cx}$ in place of $\rho$.

\begin{corollary}\label{Coro-full-L-massA}
Assume that 
\begin{equation}\label{jbjsh01A}
 \rho\le\varrho_p(\rho)  \quad \mbox{and}\quad \{y_i\}\subset \mathbb{U}\left(o,\frac{1}{3}\rho\right),
\end{equation}
and $\{y_i\}$ is   not colinear if  $p=1$.
Let $x_0\in \mathbb{U}(o,\frac{1}{3}\rho)$ be such that \eqref{indif}
holds for $p\in [1,2)$. 
Then, Algorithm {\rm\ref{GA}} for solving problem \eqref{OP-LP-def} 
with initial point $x_0$ and $t_0\in \left(0,\frac{2}{\lambda_p(x_0)}\right)$
is well defined and converges linearly to the unique Riemannian $L^p$ center of mass of $\{y_i\}$. 
\end{corollary}

\begin{proof} Note that  \eqref{jbjsh01} holds by   \eqref{jbjsh01A}. To apply Corollary \ref{Coro-full-L-mass-1}, we only need to show \eqref{gcwdinu0}. To do this, let  $z\in M\setminus\mathbb{U}(o,\rho)$. Then, we have by \eqref{jbjsh01A} and the choice of $x_0$  that
 ${\rm d}(x_0,y_i)<\frac{2\rho}{3}<{\rm d}(z,y_i)$ for each $i\in I$;
hence $f_p(x_0)<f_p(z)$ by definition (see \eqref{LP-def}). This means that $z\notin L^0_{f_p}$, establishing  \eqref{gcwdinu0} as $z\in M\setminus\mathbb{U}(o,\rho)$ is arbitrary. Thus, Corollary \ref{Coro-full-L-mass-1} is applicable to completing the proof.
\end{proof}




In the special case when $M$ is a Hadamard manifold, one checks by definition (see \eqref{Delta-p}) that $\varrho_p(r)=+\infty$ for each $r>0$.
Then, we can choose that $\rho:=+\infty$  so that  \eqref{jbjsh01} and \eqref{jbjsh01A} hold trivially. Thus, Corollary \ref{THA-full-L-masscH} follows
direct from
Corollaries \ref{THA-full-L-massc} and \ref{Coro-full-L-massA}.

%

\begin{corollary}\label{THA-full-L-masscH}
Assume that $M$ is a Hadamard manifold and $\{y_i\}$ is not colinear for  $p=1$,  and let $x_0\in M$. Then, the following assertions hold:

{\rm (i)} If \eqref{indif} holds for $p=1$, then  Algorithm {\rm\ref{GDA}} for solving problem \eqref{OP-LP-def} employing the Armijo step sizes  with initial point $x_0$ is well defined and the generated sequence $\{x_k\}$ converges to the unique Riemannian $L^p$ center of mass of $\{y_i\}$; and the
convergence rate is at least linear if \eqref{ninnonde-asm} is additionally  assumed for $p\in(1,2)$.

{\rm (ii)} If \eqref{indif} holds for $p\in[1,2)$ and $t_0\in \left(0,\frac{2}{\lambda_p(x_0)}\right)$, then Algorithm {\rm\ref{GA}} for solving problem \eqref{OP-LP-def}
 with initial point $x_0$ is well defined and converges linearly to the unique Riemannian $L^p$ center of mass of $\{y_i\}$. 
\end{corollary}



\end{CJK*}

\end{document}